\documentclass[12pt, reqno]{amsart}
\usepackage{amsmath, amsthm, amscd, amsfonts, amssymb, graphicx, color, mathrsfs}
\usepackage[bookmarksnumbered, colorlinks, plainpages]{hyperref}
\usepackage{mathscinet}
\hypersetup{colorlinks=true,linkcolor=red, anchorcolor=green, citecolor=cyan, urlcolor=red, filecolor=magenta, pdftoolbar=true}

\newtheorem{theorem}{Theorem}[section]
\newtheorem{lemma}[theorem]{Lemma}
\newtheorem{proposition}[theorem]{Proposition}

\theoremstyle{definition}
\newtheorem{definition}[theorem]{Definition}
\newtheorem{example}[theorem]{Example}

\theoremstyle{remark}
\newtheorem{remark}[theorem]{Remark}
\numberwithin{equation}{section}

\renewcommand{\AA}{\mathcal{A}}
\newcommand{\BB}{\mathcal{B}}

\newcommand{\DD}{\mathcal{D}}
\newcommand{\EE}{\mathcal{E}}
\newcommand{\FF}{\mathcal{F}}

\newcommand{\LL}{\mathcal{L}}

\newcommand{\PP}{\mathcal{P}}

\renewcommand{\SS}{\mathscr{S}}

\newcommand{\field}[1]{\mathbb{#1}}

\newcommand{\R}{\field{R}}
\newcommand{\N}{\field{N}}

\newcommand{\E}{\field{E}}

\renewcommand{\P}{\field{P}}

\newcommand{\supp}{{\rm supp}}

\newcommand{\Lip}{{\rm{Lip}}}
\newcommand{\Lipb}{\Lip_{\rm b}}

\newcommand{\Con}{{\rm{C}}}
\newcommand{\Cb}{\Con_{\rm b}}
\newcommand{\Cka}{\Con_\kappa}

\newcommand{\al}{\alpha}
\newcommand{\be}{\beta}

\newcommand{\de}{\delta}
\newcommand{\ep}{\varepsilon}

\newcommand{\ka}{\kappa}
\newcommand{\la}{\lambda}

\newcommand{\si}{\sigma}

\newcommand{\ph}{\varphi}

\newcommand{\om}{\omega}

\newcommand{\De}{\Delta}

\newcommand{\La}{\Lambda}
\newcommand{\Si}{\Sigma}

\newcommand{\Om}{\Omega}

\newcommand{\sm}{\setminus}
\newcommand{\es}{\emptyset}

\newcommand{\UC}{\mathop{\text{\upshape{UC}}}\nolimits}

\newcommand{\tr}{\mathop{\text{\upshape{tr}}}\nolimits}

\begin{document}

\setcounter{page}{1}

\title[Upper envelopes and viscosity solutions]{Upper envelopes of families of Feller semigroups and viscosity solutions to a class of nonlinear Cauchy problems}

\author[M.~Nendel \MakeLowercase{and} M.~R\"ockner]{Max Nendel$^1$\MakeLowercase{and} Michael R\"ockner$^2$}

\address{$^{1}$Center for Mathematical Economics, Bielefeld University, 33615 Bielefeld, Germany}
\email{\textcolor[rgb]{0.00,0.00,0.84}{Max.Nendel@uni-bielefeld.de}}

\address{$^{2}$Faculty of Mathematics, Bielefeld University, 33615 Bielefeld, Germany}
\email{\textcolor[rgb]{0.00,0.00,0.84}{Roeckner@math.uni-bielefeld.de}}

\date{\today}

\thanks{Financial support through the German Research Foundation via CRC 1283 ``Taming Uncertainty'' is gratefully acknowledged. The authors thank two anonymous referees for many comments and suggestions that lead to a decisive improvement in the presentation of the manuscript, as well as Liming Yin for his helpful observations.}


\begin{abstract}
In this paper, we consider the (upper) semigroup envelope, i.e.\ the least upper bound, of a given family of linear Feller semigroups. We explicitly construct the semigroup envelope and show that, under suitable assumptions, it yields viscosity solutions to abstract Hamilton-Jacobi-Bellman-type partial differential equations related to stochastic optimal control problems arising in the field of Robust Finance. We further derive conditions for the existence of a Markov process under a nonlinear expectation related to the semigroup envelope for the case where the state space is locally compact. The procedure is then applied to numerous examples, in particular, nonlinear PDEs that arise from control problems for infinite dimensional Ornstein-Uhlenbeck and L\'evy processes.

\smallskip
\noindent \emph{Key words:} Semigroup envelope, fully nonlinear PDE, viscosity solution, Feller process, model uncertainty, nonlinear expectation

\smallskip
\noindent \emph{AMS 2010 Subject Classification:} 47H20; 49L25; 60G20
\end{abstract}

\maketitle

\section{Introduction}

Assume that we are given a ``nice'' Feller process and that there are some features, for example some parameters (drift, volatility, etc.), of the process that cannot be determined precisely. In this case, one typically speaks of model uncertainty or ambiguity. This topic has been studied extensively in the context of Economics and Mathematical Finance in the last decades. Prominent examples include a Brownian motion (Bachelier model) with drift uncertainty (cf.~Coquet et al.~\cite{MR1906435}) or volatility uncertainty (cf. Peng~\cite{PengG},\cite{MR2474349}), a Black-Scholes model with volatility uncertainty (cf. Epstein and Ji \cite{Epji}, Vorbrink \cite{MR3250653}), and L\'evy processes with uncertainty in the L\'evy triplet (cf. Hu and Peng \cite{PengHu}, Neufeld and Nutz \cite{NutzNeuf}, Hollender \cite{H2016}, K\"uhn \cite{MR3941868}, Denk et al.~\cite{dkn}). Under this type of uncertainty, worst case considerations together with dynamic consistency requirements lead to a stochastic optimal control problem, where, intuitively speaking, ``nature'' tries to control the system into the worst possible scenario, and to the consideration of so-called nonlinear expectations. In the case of a Brownian Motion with uncertain volatility within an interval $[\si_\ell,\si_h]$ with $0<\si_\ell<\si_h$, this leads, for instance, to the control problem
\begin{equation}\label{optcont}
 V(t,x;u_0):=\sup_{\sigma \in \Sigma} \E\bigg[u_0\bigg(x+\int_0^t \sigma_s\, {\rm d}B_s\bigg)\bigg],
\end{equation}
where $B$ is a standard Brownian Motion on a suitable filtered probability space and $\Sigma$ consists of all progressively measurable stochastic processes $\si=(\si_t)_{t\geq 0}$ with values in $[\si_\ell,\si_h]$. Solving the optimal control problem \eqref{optcont} then results in the HJB equation
\[
  \partial_t u(t,x)=\sup_{\sigma\in [\si_\ell,\si_h]} \frac{\sigma^2}{2}\partial_{xx}u(t,x)\quad \text{for }t\geq 0\text{ and }x\in \R, \quad u(0)=u_0,
\]
which is typically referred to as $G$-heat equation. We refer to Denis et al. \cite{MR2754968} for a detailed illustration of this relation. Moreover, one can show that the value function \eqref{optcont} admits a representation of the form
\[
 V(t,x;u_0)=\EE\big(u_0(x+X_t)\big),
\]
where $\EE$ is a sublinear expectation, more precisely a $G$-expectation, and $X$ is a so-called $G$-Brownian Motion (cf.~Denis et al. \cite{MR2754968} and Peng~\cite{PengG},\cite{MR2474349}).\\

Motivated by this example, we choose a semigroup-theoretic approach, formally separating the space and time variable, in order to prove the existence of viscosity solutions to abstract Hamilton-Jacobi-Bellman-type equations of the form 
\begin{equation}\label{nonlinPDE}
 \partial_t u(t)=\sup_{\la\in \La} A_\la u(t)\quad \text{for }t\geq 0, \quad u(0)=u_0, 
\end{equation}
where $(A_\lambda)_{\lambda\in \Lambda}$ is a family of generators of Feller processes indexed by a nonempty index set $\Lambda$. We refer to Engel and Nagel \cite{MR1721989} or Pazy \cite{MR710486} for more details on semigroup theory related to linear PDEs and the idea of formally separating space and time. Our approach is based on an explicit construction and approximation of the solution due to Nisio \cite{Nisio}, which adds a primal description to the dual representation in terms of a stochastic optimal control problem. In a second step, we discuss how a stochastic process under a sublinear expectation can be obtained from the nonlinear semigroup which describes the transition of the process, using a nonlinear version of Kolmogorov's extension theorem by Denk et al. \cite{MR3824739}. Finally, we link semigroup envelopes to the value functions of abstract versions of Meyer-type control problems. We thus provide a nonlinear analogue to the classical relation between Feller processes, partial differential equations and semigroups. It is worth noting that stochastic optimal control problems and nonlinear PDEs of the form \eqref{nonlinPDE} are intimately related to BSDEs (cf. Pardoux and Peng \cite{MR1176785},\cite{MR1037747}, El Karoui et al. \cite{MR1434407}, Coquet et al.~\cite{MR1906435}), 2BSDEs (cf. Cheridito et al. \cite{MR2319056}, Soner et al. \cite{MR2746175},\cite{MR2925572}) and BSDEs with jumps (cf. Kazi-Tani et al.~\cite{MR3361253},\cite{MR3375890}) resulting in a stochastic representation of solutions to nonlinear Cauchy problems of the form \eqref{nonlinPDE}. The present paper can be seen as an analytic counter part to these approaches, which are based on mainly stochastic methods, and the techniques we use might pave the way for further applications in control theory.\\

For two (possibly nonlinear) semigroups $S=(S(t))_{t\geq 0}$ and $T=(T(t))_{t\geq 0}$ on a Banach lattice $X$, we write $S\leq T$ if $S(t) x\leq T(t)x$ for all $t\geq 0$ and $x\in X$. For a nonempty index set $\La$ and a family $(S_\la)_{\la\in \La}$ of semigroups on $X$ we call a semigroup $T$ an upper bound of $(S_\la)_{\la\in \La}$ if $T\geq S_\la$ for all $\la\in \La$. We call $\SS$ the least upper bound of $(S_\la)_{\la\in \La}$ if $\SS$ is an upper bound of $(S_\la)_{\la\in \La}$ and $\SS\leq T$ for any other upper bound $T$ of $(S_\la)_{\la\in \La}$. Then, the question arises under which conditions the family $(S_\la)_{\la\in \La}$ has a least upper bound. To the best of our knowledge this question has first been addressed by Nisio \cite{Nisio}, in the case every $S_\lambda$ is a strongly continuous semigroup on the space of all bounded measurable functions, which is why we call the least upper bound $\SS$ of $(S_\la)_{\la\in \La}$ the \textit{Nisio semigroup} or the \textit{(upper) semigroup envelope} of $(S_\la)_{\la\in \La}$. Due to a Theorem of Lotz \cite{MR797538} it is known that strongly continuous linear semigroups on the space of all bounded measurable functions always have a bounded generator, which is why the result of Nisio is not applicable for most semigroups related to partial differential equations. However, using a similar approach to the one by Nisio on the space of bounded and uniformly continuous functions, Denk et al.~\cite{dkn} proved the existence of a least upper bound for transition semigroups of L\'evy processes. In the present paper, we use the idea of Nisio in a more general framework than Denk et al.~\cite{dkn} in order to go beyond L\'evy processes. Main examples will be transition semigroups of Ornstein-Uhlenbeck processes and L\'evy processes on real separable Hilbert spaces, Geometric Brownian Motions, and Koopman semigroups with semiflows in real separable Banach spaces.\\

A fundamental result from semigroup theory is the fact that for a strongly continuous semigroup $S=(S(t))_{t\geq 0}$ of linear operators with generator $A$ the function $u(t):=S(t)u_0$, for sufficiently regular initial data $u_0$, is a solution to the abstract Cauchy problem
\begin{equation}\label{linPDE}
 \partial_t u(t)=A u(t)\quad \text{for }t\geq 0, \quad u(0)=u_0.
\end{equation}
We refer to Engel and Nagel \cite{MR1721989} or Pazy \cite{MR710486} for more details on this relation. Similar as in the work by Denk et al.~\cite{dkn}, we show that the semigroup envelope yields a viscosity solution to the nonlinear Cauchy problem \eqref{nonlinPDE} if $A_\la$ is the generator of $S_\la$ for all $\la\in \La$. On one side, this is interesting from a structural point of view, since it establishes a relation between the least upper bound of a family of semigroups and the least upper bound of their generators. On the other side, this shows that semigroup envelopes are closely related to solutions to possibly infinite-dimensional stochastic optimal control problems as well as local and non-local Hamilton-Jacobi-Bellman equations in Hilbert spaces, cf.~Barbu and Da Prato \cite{MR623937},\cite{MR641828},\cite{MR704182},\cite{MR839043}, Fabbri et al.~\cite{MR3674558}, Federico and Gozzi \cite{MR3861820}, \'{S}wi\polhk ech and Zabczyk \cite{MR3010779},\cite{MR3558357}. We point out that, in comparison to the standard literature on control theory and viscosity theory, our approach covers a different spectrum of applications. While in the standard theory on viscosity solutions very general types of HJB equations of the form
\[
 u_t=F\big(t,x,u(t,x),D_x u(t,x), D_{xx}u(t,x)\big)
\]
with a suitable function $F$ are considered, our approach uses very much the particular structure of the equation \eqref{nonlinPDE}. On the other hand, we allow for very general forms of generators, which are not covered by standard results. However, as we discuss in Section \ref{sec:control}, in most cases that are covered by, both, the standard approach and our approach, the solution concepts coincide. We thus propose a different yet consistent solution concept, which allows to cover a different range of examples, in particular, completely non-standard control problems. In order to come up with control problems that are somewhat closer to reality, in the past decades, an increasing interest has been paid to infinite-dimensional control problems with a particular focus on infinite-dimensional controlled Ornstein-Uhlenbeck processes. We refer to Fabbri et al.~\cite{MR3674558} and the references therein for a detailed discussion on this topic. Considering a family $(A_\lambda)_{\lambda\in \Lambda}$ of generators of infinite-dimensional Ornstein-Uhlenbeck processes, we cover a certain range of examples for Ornstein-Uhlenbeck control problems. In the standard theory on controlled Ornstein-Uhlenbeck processes (cf. Fabbri et al.~\cite{MR3674558}) the drift term consists of an expression of the form $\big(BX_t+m \big){\rm d}t$ with a fixed unbounded generator $B$ and a controlled vector $m$. Under certain conditions, the existence of mild solutions and $C^1$-regularity of the related HJB equation can be obtained using smoothing properties of the linear semigroup related to $B$ and perturbation results from semigroup theory for semilinear equations. Our approach allows to consider controlled Ornstein-Uhlenbeck processes with bounded generators in the drift term with controls in terms of $B$, $m$ and the covariance operator in the diffusion part (see Example \ref{ex:OU}). In a forthcoming paper with Ben Goldys and the authors we show that our approach also extends to unbounded operators $B$

Throughout, we consider a nonempty index set $\La$, a fixed separable metric space $(M,d)$ and a fixed weight function $\kappa\colon M\to (0,\infty)$, which is assumed to be continuous and bounded. Let $\Con=\Con (M)$ be the space of all continuous functions $M\to \R$. We denote the space of all $u\in \Con$ with norm
\[
 \|u\|_\infty:=\sup_{x\in M}|u(x)|<\infty
\]
by $\Cb$ and the space of all $u\in \Con$ with seminorm
\[
 \|u\|_{\Lip}:=\inf\big\{L\geq 0\, |\, \forall x,y\in M:|u(x)-u(y)|\leq Ld(x,y)\big\}<\infty
\]
by $\Lip$. Finally, we denote the space of all $u\in C$ with norm
\[
 \|u\|_\kappa:=\|\kappa u\|_\infty<\infty
\]
by $\Cka$ and the closure of $\Lipb:=\Lip\cap \Cb$ in the space $\Cka$ by $\UC_\kappa$. If $\ka$ is bounded below by some positive constant, then $C_\ka=C_{\rm b}$ and $\|\cdot \|_\ka$ is equivalent to $\|\cdot \|_\infty$. In this case, $\UC_\ka$ is the closure of $\Lipb$ w.r.t. $\|\cdot\|_\infty$, which is the space $\UC_{\rm b}$ of all bounded and uniformly continuous functions $M\to \R$. If $M$ has the Heine-Borel property, i.e if every closed bounded subset of $M$ is compact, and $\ka\in \Con_0$, then $\UC_\kappa=\{u\in \Con\, |\, \kappa u\in \Con_0\}$, where $\Con_0$ is the closure of the space $\Con_c$ of all continuous functions with compact support w.r.t.\ $\|\cdot\|_\infty$. We refer to Example \ref{rem.contabove} b) for more details. For a sequence $(u_n)_{n\in \N}\subset \UC_\kappa$ and $u\in \UC_\kappa$, we write $u_n\nearrow u$ as $n\to \infty$ if $u_n\leq u_{n+1}$ for all $n\in \N$ and $u_n(x)\to u(x)$ as $n\to \infty$ for all $x\in M$. Analogously, we write $u_n\searrow u$ as $n\to \infty$ if $u_n\geq u_{n+1}$ for all $n\in \N$ and $u_n(x)\to u(x)$ as $n\to \infty$ for all $x\in M$. We are now ready to introduce the central objects of our discussion.

\begin{definition}\label{def:feller}\
 \begin{enumerate}
  \item[a)] We call a family $\SS=(\SS(t))_{t\geq 0}$ of possibly nonlinear operators a \textit{Feller semigroup} if the following conditions are satisfied:
  \begin{enumerate}
   \item[(i)] $\SS(t)\colon \UC_\kappa \to \UC_\kappa$ is continuous for all $t\geq 0$,
   \item[(ii)] $\SS(0)u=u$ and $\SS(s+t)u=\SS(s)\SS(t)u$ for all $s,t\geq 0$ and $u\in \UC_\kappa$,
   \item[(iii)] $\SS(t)$ is \textit{monotone} and \textit{continuous from below} for all $t\geq 0$, i.e. for any sequence $(u_n)_{n\in \N}\subset \UC_\kappa$ and $u\in \UC_\kappa$ with $u_n\nearrow u$ as $n\to \infty$ it holds $\SS(t) u_n\nearrow \SS(t) u$ as $n\to \infty$.
  \end{enumerate}
  \item[b)] Let $D\subset \UC_\ka$. We then say that a Feller semigroup $\SS$ is \textit{strongly continuous} on $D$ if the map
  \[
   [0,\infty)\to \UC_\ka,\quad  t\mapsto \SS(t)u
  \]
  is continuous for all $u\in D$. If $D=\UC_\ka$, we say that $\SS$ is \textit{strongly continuous}.
 \end{enumerate}
\end{definition}

Note that our definition of a Feller semigroup is somewhat different from the standard notion in the literature. First of all, we do not require strong continuity or linearity of the semigroup a priori, as it is usually the case. Moreover, Feller semigroups are oftentimes related to functions vanishing at infinity. In order to treat situations, where the state space $M$ is infinite-dimensional, we do not require any condition related to compact sets but rather a certain growth condition in terms of the weight function $\kappa$.

Throughout this work, we assume the following setup:
\begin{enumerate}
 \item[(A1)] For all $\la\in \La$ let $S_\la$ be a Feller semigroup of linear operators with $S_\la (t)1=1$, where $1$ denotes the constant $1$-function.
 \item[(A2)] There exist constants $\al,\be\in \R$ such that
 \[
 \|S_\la(t)u\|_\kappa\leq e^ {\al t}\|u\|_\kappa \quad \text{and}\quad \|S_\la(t)u\|_{\Lip}\leq e^{\be t}\|u\|_{\Lip}
 \]
 for all $u\in \Lipb$, $\la\in \La$ and $t\geq 0$.
\end{enumerate}

At this point, we would like to briefly discuss the assumptions (A1) and (A2) and explain the key differences between the present paper and the paper by Denk et al.~\cite{dkn}. First, we would like to mention that the assumptions (A1) and (A2) are satisfied with $\kappa=1$, $\alpha =0$ and $\beta=0$ for Markovian convolution semigroups (semigroups arising from L\'evy processes). Different from \cite{dkn}, we do not make any assumption on strong continuity of the semigroups $(S_\la)_{\la\in \La}$ or their generators at this point. Strong continuity was a key ingredient in the proof of the dynamic programming principle (the semigroup property of the semigroup envelope) in \cite{dkn} and also in the paper by Nisio \cite{Nisio}. In this paper, we provide an alternative proof for the dynamic programming principle, which does not require any strong continuity assumptions, and covers a more general setup. In particular, we prove the existence of the semigroup envelope of the family $(S_\la)_{\la\in \La}$ (Theorem \ref{main1}) solely under the assumptions (A1) and (A2). In Section \ref{sec:strongcont}, we then provide three conditions that imply the strong continuity of the Nisio semigroup, which in turn implies that the Nisio semigroup is a viscosity solution to a nonlinear Cauchy problem (cf. Section \ref{sec:viscosity}). The key assumption in order to obtain the strong continuity in \cite{dkn} and \cite{Nisio} is a joint density assumption on the domains of the generators, which, in some infinite-dimensional applications, is not satisfied. In particular, uncertainty in the covariance operator of infinite-dimensional Brownian Motions leads to major restrictions, see \cite[Example 3.3]{dkn}. The conditions for strong continuity and the generalised setup, we present in this paper, allow us to treat, both, finite and infinite-dimensional applications (Koopman semigroups, geometric dynamics, Ornstein-Uhlenbeck processes and L\'evy processes) in full generality concerning the uncertainty, and to improve \cite[Example 3.3]{dkn} in such a way that no L\'evy triplet is excluded a priori. The assumption in order to obtain the strong continuity in \cite{dkn} is a special case of Proposition \ref{scontin} in the present paper. Finally, we would like to point out that the setup we choose is also more flexible regarding the tail behaviour of solutions. More precisely, the choice of the weight function $\kappa$ enables us to consider also unbounded initial data (contingent claims), which was not possible in the setup chosen by Denk et al.\\
 
The paper is structured as follows. In Section \ref{sec:proof}, we show the existence of the semigroup envelope $\SS$ of the family $(S_\la)_{\la\in \La}$ under the assumptions (A1) and (A2), and provide approximation results for the Nisio semigroup. The main result of this section is Theorem \ref{main1}. In Section \ref{sec:strongcont}, we provide conditions that guarantee the strong continuity of the semigroup envelope (Propositions \ref{critstrongcont1} -  \ref{critstrongcont2}). In Section \ref{sec:viscosity}, we discuss the connection between semigroup envelopes and viscosity solutions to a nonlinear abstract Cauchy problem. The main result of this section is Theorem \ref{viscosity1}. In Section \ref{sec:expec}, we give a stochastic representation of the semigroup envelope via a stochastic process under a sublinear expectation (cf. Theorem \ref{stochrep}). Section \ref{sec:control} is devoted to the connection between the results obtained in the present paper and the field of control theory. In particular, we explain the link between semigroup envelopes and value functions of abstract control problems. In Section \ref{sec:ex}, we apply the results from Sections \ref{sec:proof}, \ref{sec:strongcont} and \ref{sec:expec} to several non-standard examples.

\section{Construction of the semigroup envelope}\label{sec:proof}

Let $u\in \UC_\ka$, $\la\in \La$ and $h\geq 0$. Then, $\|S_\la(h)u\|_\ka\leq e^{\al h}\|u\|_\ka$ since the map $S_\la(h)\colon \UC_\ka\to \UC_\ka$ is continuous, which implies that
\[
 \big(\EE_h u\big)(x):=\sup_{\la\in \La} \big(S_\la(h)u\big)(x)
\]
is well-defined for all $x\in M$.

\begin{lemma}\label{lip121}
 Let $h\geq 0$.
 \begin{enumerate}
  \item[a)] $\|\EE_h u-\EE_h v\|_\ka\leq e^{\al h} \|u-v\|_\ka$ for all $u,v\in \UC_\ka$.
  \item[b)] $\|\EE_h u\|_{\Lip}\leq e^{\be h}\|u\|_{\Lip}$ for all $u\in \Lipb$. 
  \item[c)] The map $\EE_h\colon \UC_\ka\to \UC_\ka$ is well-defined and Lipschitz continuous with Lipschitz constant $e^{\al h}$.
  \item[d)] $\EE_h$ is sublinear, monotone, and continuous from below with $\EE_h 1=1$.
 \end{enumerate}
\end{lemma}

\begin{proof}\
 \begin{enumerate}
  \item[a)] Let $u,v\in \UC_\ka$ and $h\geq 0$. Then, for all $\la\in \La$,
  \begin{align*}
   \ka\big(S_\la(h)u-\EE_h v\big)&\leq \ka\big(S_\la(h)u-S_\la(h)v\big)= \ka S_\la(h)(u-v)\\
   &\leq \|S_\la(h)(u-v)\|_\ka\leq e^{\al h}\|u-v\|_\ka.
  \end{align*}
 Taking the supremum over $\la\in \La$ and a symmetry argument imply that
  \[
   \|\EE_h u-\EE_h v\|_\ka\leq e^{\al h}\|u-v\|_\ka.
  \]
  \item[b)] Let $u\in \Lipb$ and $x,y\in M$. Then, for all $\la\in \La$,
  \[
   (S_\la(h)u\big)(x)-\big(\EE_h u\big)(y)\leq (S_\la(h)u\big)(x)-(S_\la(h)u\big)(y)\leq e^{\be h}\|u\|_{\Lip}d(x,y).
  \]
  Taking the supremum over $\la\in \La$ and a symmetry argument yield that
  \[
   \big|\big(\EE_h u\big)(x)-\big(\EE_h u\big)(y)\big|\leq e^{\be h}\|u\|_{\Lip}d(x,y).
  \]
  \item[c)] By part b) and Assumption (A1), we have that $\EE_h u\in \Lipb$ for all $u\in \Lipb$. Since $\Lipb$ is dense in $\UC_\ka$, part a) implies that $\EE_h\colon \UC_\ka\to \UC_\ka$ is well-defined and Lipschitz continuous with Lipschitz constant $e^{\al h}$.
  \item[d)] All these properties directly carry over to the supremum.
 \end{enumerate}
\end{proof}

In the sequel, we consider the set $P:= \{\pi\subset[0,\infty)\colon 0\in\pi, \, |\pi|<\infty\}$ of finite partitions of the positive half line. The set of partitions with end-point $t$ will be denoted by $P_t$, i.e. $P_t := \{\pi \in P: \max \pi = t\}$. Let $u\in \UC_\kappa$ and $\pi\in P\setminus\big\{\{0\}\big\}$. Then, there exist $0=t_0<t_1<\ldots <t_m$ such that $\pi=\{t_0,t_1,\dots,t_m\}$ and we set
\[
 \EE_\pi u := \EE_{t_1-t_0} \ldots \EE_{t_m-t_{m-1}} u.
\]
Moreover, we set $\EE_{\{0\}}u := u$. Note that, by definition, $\EE_h = \EE_{\{0,h\}}$ for $h>0$. Since $\EE_h\colon \UC_\ka\to \UC_\ka$ is well-defined, the map $\EE_\pi\colon \UC_\ka\to \UC_\ka$ is well-defined, too.

\begin{lemma}\label{lip12}
 For all $\pi\in P$, the operator $\EE_\pi$ is sublinear, monotone and continuous from below with $\EE_\pi 1=1$. Moreover, $\|\EE_\pi u-\EE_\pi v\|_\ka\leq e^{\al \max \pi} \|u-v\|_\ka$ for all $u,v\in \UC_\ka$ and $\|\EE_\pi u\|_{\Lip}\leq e^{\be \max \pi}\|u\|_{\Lip}$ for all $u\in \Lipb$.
\end{lemma}

\begin{proof}\
 Since $\EE_h$ is a sublinear, monotone and continuous from below with $\EE_h1=1$ for all $h\geq 0$, the same holds for $\EE_\pi$ as these properties are preserved under compositions. The Lipschitz continuity follows from Lemma \ref{lip121} and the behaviour of Lipschitz constants under composition.
\end{proof}

Let $u\in \UC_\ka$. In the following, we consider the limit of $\EE_\pi u$ when the mesh size of the partition $\pi\in P$ tends to zero. First note that, for $h_1,h_2\geq 0$ and $x\in M$,
\begin{align*}
 \big(\EE_{h_1+h_2}u\big)(x)&=\sup_{\la\in \La} \big(S_\la(h_1+h_2)u\big)(x)=\sup_{\la\in \La} \big(S_\la(h_1)S_\la(h_2)u\big)(x)\\
 &\leq \sup_{\la\in \La} \big(S_\la(h_1)\EE_{h_2}u\big)(x)=\big(\EE_{h_1}\EE_{h_2}u\big)(x),
\end{align*}
which implies the pointwise inequality
\begin{equation}\label{monotone}
 \EE_{\pi_1}u\leq \EE_{\pi_2}u \quad\text{for }\pi_1,\pi_2\in P\text{ with }\pi_1\subset \pi_2.
\end{equation}
 In particular, for $\pi_1,\pi_2\in P$ and $\pi:=\pi_1\cup \pi_2$ it follows that $\pi\in P$ with
\begin{equation}\label{0}
 \big(\EE_{\pi_1}u\big)\vee \big(\EE_{\pi_2}u\big)\leq \EE_\pi u.
\end{equation}
Recall that we denote the set of all finite partitions with end point $t\geq 0$ by $P_t$. For $t\geq 0$, $x\in M$ and $u\in \UC_\ka$, we define
\begin{equation}\label{def:Nisio}
 \big(\SS(t)u\big)(x):=\sup_{\pi\in P_t} \big(\EE_\pi u\big)(x).
\end{equation}
The family $\SS=(\SS(t))_{t\geq 0}$ is called the \textit{(upper) semigroup envelope} or \textit{Nisio semigroup} of the family $(S_\la)_{\la\in \La}$. Note that, by definition, $\SS(0)u=u$ for all $u\in \UC_\ka$. We observe the following basic facts, which are a direct consequence of Lemma \ref{lip12}.

\begin{lemma}\label{lem:observ}
  Let $t\geq 0$. Then, the map $\SS(t)\colon \UC_\ka\to \UC_\ka$ is well-defined and Lipschitz continuous with Lipschitz constant $e^{\al t}$. Moreover, $\SS(t)$ is sublinear, monotone and continuous from below with $\SS(t)1=1$.
\end{lemma}

\begin{proof}
 By Lemma \ref{lip12},
 \begin{equation}\label{eq:lip.nisio}
 \|\SS(t) u-\SS(t) v\|_\ka\leq e^{\al t} \|u-v\|_\ka\quad \text{for all }u,v\in \UC_\ka  
 \end{equation}
 and $\|\SS(t)u\|_{\Lip}\leq e^{\be t}\|u\|_{\Lip}$ for all $u\in \Lipb$. In particular, $\SS(t)u\in \Lipb$ for all $u\in \Lipb$. Now, the estimate \eqref{eq:lip.nisio} implies that $\SS(t)\colon \UC_\ka\to \UC_\ka$ is well-defined and Lipschitz continuous with Lipschitz constant $e^{\al t}$. The remaining properties follow directly from the observation that, by Lemma \ref{lip12} they are satisfied by $\EE_\pi$, for $\pi\in P_t$, and carry over to the supremum over all $\pi\in P_t$.
 \end{proof}
 
 In the following, we show that the Nisio semigroup $\SS$ is in fact a semigroup. We start with the following lemma, which shows that $\SS(t)u$ can be approximated by a monotone sequence of partitions depending on $u$. We would like to point out that, under additional assumptions, the dependence of the sequence on $u$ can be dropped (see Proposition \ref{monlimit}, below).
 
\begin{lemma}\label{seq}
 Let $u\in \UC_\ka$ and $t>0$. Then, there exists a sequence $(\pi_n)_{n\in \N}\subset P_t$ (depending on $u$) with $\EE_{\pi_n}u\nearrow \SS(t)u$ as $n\to \infty$.
\end{lemma}

\begin{proof}
 Let $(x_k)_{k\in \N}\subset M$ such that the set $\{x_k\,|\,  k\in \N\}$ is dense in $M$. Then, for every $k\in \N$, there exists a sequence $(\pi_n^k)_{n\in \N}\subset P_t$ with $\pi_n^k\subset \pi_{n+1}^k$ for all $n\in \N$ and
 \[
  \big(\EE_{\pi_n^k}u\big)(x_k)\nearrow \big(\SS(t)u\big)(x_k)\quad \text{as } n\to \infty.
 \]
 Now, let $\pi_n:=\bigcup_{k=1}^n\pi_n^k$ for all $n\in \N$. Then, $\pi_n^k\subset \pi_n\subset \pi_{n+1}$ for all $n\in \N$ and $k\in \{1,\ldots,n\}$. Hence,
 \begin{equation}\label{proof:1}
  \EE_{\pi_n^k}u\leq \EE_{\pi_n}u\leq \EE_{\pi_{n+1}}u\quad \text{for all }n\in \N\text{ and }k\in \{1,\ldots,n\}.
 \end{equation}
  Let $\big(\EE_\infty v\big)(x):= \sup_{n\in \N}\big(\EE_{\pi_n}v\big)(x)$ for all $v\in \UC_\ka$ and $x\in M$. Then, by Lemma \ref{lip12}, the map $\EE_\infty\colon \UC_\ka\to \UC_\ka$ is well-defined. In particular, $\EE_\infty u\colon M\to \R$ is continuous and, by \eqref{proof:1}, $\EE_{\pi_n}u\nearrow \EE_\infty u$ as $n\to \infty$. Again, by \eqref{proof:1},
 \[
  \big(\SS(t)u\big)(x_k)=\lim_{n\to \infty}\big(\EE_{\pi_n^k}u\big)(x_k)\leq \lim_{n\to \infty}\big(\EE_{\pi_n}u\big)(x_k)=\big(\EE_\infty u\big)(x_k)\leq \big(\SS(t)u\big)(x_k)
 \]
 for all $k\in \N$. Since, $\SS(t)u$ and $\EE_\infty u$ are both continuous and the set $\{x_k\,|\,  k\in \N\}$ is dense in $M$, it follows that $\SS(t)u= \EE_\infty u$, which shows that
 \[
  \EE_{\pi_n}u\nearrow \SS(t)u \quad \text{as } n\to \infty.
 \]
\end{proof}

We obtain the following main theorem.

\begin{theorem}\label{main1}
 The family $\SS$ is a Feller semigroup of sublinear operators and the least upper bound of the family $(S_\la)_{\la\in \La}$.
\end{theorem}

\begin{proof}
We first show that, for all $s,t\geq 0$,
\begin{equation}\label{dpp1}
 \SS(s+t)=\SS(s)\SS(t).
\end{equation}
 If $s=0$ or $t=0$ the statement is trivial. Therefore, let $s,t>0$, $u\in \UC_\ka$, $\pi_0\in P_{s+t}$, and $\pi:=\pi_0\cup \{s\}$. Then, $\pi\in P_{s+t}$ with $\pi_0\subset \pi$ and, by \eqref{monotone}, $\EE_{\pi_0}u\leq \EE_{\pi}u$. Let $m\in \N$, $0=t_0<t_1<\ldots t_m=s+t$ with $\pi=\{t_0,\ldots, t_m\}$, and $i\in \{1,\ldots, m\}$ with $t_i=s$. Then, $\pi_1:=\{t_0,\ldots, t_i\}\in P_s$ and $\pi_2:=\{t_i-s,\ldots, t_n-s\}\in P_t$ with
 \[
  \EE_{\pi_1}=\EE_{t_1-t_0}\cdots \EE_{t_i-t_{i-1}}\quad \text{and}\quad \EE_{\pi_2}=\EE_{t_{i+1}-t_i}\cdots \EE_{t_m-t_{m-1}}.
 \]
 We thus obtain that
 \begin{align*}
  \EE_{\pi_0}u&\leq \EE_\pi u=\EE_{t_1-t_0}\cdots \EE_{t_m-t_{m-1}}u=\big(\EE_{t_1-t_0}\cdots \EE_{t_i-t_{i-1}}\big)\big(\EE_{t_{i+1}-t_i}\cdots \EE_{t_m-t_{m-1}}u\big)\\
  &=\EE_{\pi_1}\EE_{\pi_2}u\leq \EE_{\pi_1}\SS(t)u\leq \SS(s)\SS(t)u.
 \end{align*}
 Taking the supremum over all $\pi_0\in P_{s+t}$ yields that $\SS(s+t)u\leq \SS(s)\SS(t)u$.
 
 Now, let $(\pi_n)_{n\in \N}\subset P_t$ with $\EE_{\pi_n}u\nearrow \SS(t)u$ as $n\to \infty$ (see Lemma \ref{seq}) and fix $\pi_0\in P_s$. Then, for all $n\in \N$,
 \[
  \pi_n':=\pi_0\cup \{s+\tau\colon \tau \in \pi_n\}\in P_{s+t}\quad \text{with}\quad \EE_{\pi_n'}=\EE_{\pi_0}\EE_{\pi_n}.
 \]
 As $\EE_{\pi_0}$ is continuous from below, it follows that
 \[
  \EE_{\pi_0}\big(\SS(t)u\big)=\lim_{n\to \infty}\EE_{\pi_0}\EE_{\pi_n}u=\lim_{n\to \infty}\EE_{\pi_n'}u\leq \SS(s+t)u.
 \]
 Taking the supremum over all $\pi_0\in P_s$, we get that $\SS(s)\SS(t)u\leq \SS(s+t)u$, and therefore \eqref{dpp1} follows.

 From the definition of $\SS$ in Equation \eqref{def:Nisio} and Lemma \ref{lem:observ}, we now may conclude that $\SS$ defines a Feller semigroup of sublinear operators. It remains to show that $\SS$ is the least upper bound of the family $(S_\la)_{\la\in \La}$. To this end, let $u\in \UC_\ka$, $x\in M$, and $T$ be an upper bound of the family $(S_\la)_{\la\in \La}$, i.e $\big(S_\la(t)u\big)(x)\leq \big(T(t)u\big)(x)$ for all $\la \in \La$, $u\in \UC_\ka$, $t\geq 0$ and $x\in M$. Then,
 \[
  \big(S_\la (h)u\big)(x)\leq \big(\EE_h u\big)(x) \leq \big(T(h)u\big)(x)\quad\text{for all }\la \in \La\text{ and }h\geq 0.
 \]
  Since $S_\la$ and $T$ are semigroups, it follows that
 \[
  \big(S_\la (t)u\big)(x)\leq \big(\EE_\pi u\big)(x) \leq \big(T(t)u\big)(x)\quad\text{for all }\la \in \La,\,t\geq 0,\text{ and }\pi\in P_t. 
 \]
 Taking the supremum over all $\pi\in P_t$, we obtain that
  \[
  \big(S_\la (t)u\big)(x)\leq \big(\SS(t) u\big)(x) \leq \big(T(t)u\big)(x)\quad \text{for all }\la \in \La\text{ and }t\geq 0.
 \]
\end{proof}

The remainder of this section is devoted to show that the approximation result of Lemma \ref{seq}, where the approximating sequence was dependent on the function $u\in \UC_\ka$, can be made stronger under the additional assumption that the map
\begin{equation}\label{ass.approx}
 [0,\infty)\to \UC_\ka,\quad h\mapsto \EE_h u
\end{equation}
is continuous for all $u\in \UC_\ka$. More precisely, under this condition every sequence of partitions with mesh size tending to $0$ can be used for the approximation of the semigroup envelope. Note that \eqref{ass.approx} is, for example, implied by the condition that
\[
 \sup_{\la\in \La}\|S_\la(h)u-u\|_\ka\to 0\quad \text{as }h\to 0,
\]
for all $u\in \Lipb$, which, in most applications, is satisfied. The following lemma shows that $\EE_\pi$ depends continuously on the partition $\pi\in P$.

\begin{lemma}\label{strongcontJ}\
 Assume that the map \eqref{ass.approx} is continuous for all $u\in \UC_\ka$. Let $m\in \N$ and $\pi=\{t_0,t_1,\ldots, t_m\}\in P$ with $0=t_0<\ldots <t_m$. For each $n\in \N$ let $\pi_n=\{t_0^n,t_1^n,\ldots, t_m^n\}\in P$ with $0=t_0^n<t_1^n<\ldots < t_m^n$ and $t_i^n\to t_i$ as $n\to \infty$ for all $i\in \{1,\ldots, m\}$. Then, for all $u\in \UC_\ka$ we have that
 \[
  \|\EE_{\pi}u-\EE_{\pi_n}u\|_\ka\to 0,\quad n\to \infty.
 \]
\end{lemma}

\begin{proof}
 First note that the set of all partitions with cardinality $m+1$ can be identified with the set
 \[
  S^m:=\big\{(s_1,\ldots, s_m)\in \R^m\, \big|\, 0<s_1<\ldots<s_m\big\}\subset \R^m.
 \]
 Therefore, the assertion is equivalent to the continuity of the map
 \begin{equation}\label{contmap}
  S^m\to \UC_\ka,\quad (s_1,\ldots, s_m)\to \EE_{\{0,s_1,\ldots, s_m\}}u.
 \end{equation}
 Since the mapping $[0,\infty)\to \UC_\ka,\; h\mapsto \EE_h u$ is continuous for all $u\in \UC_\ka$, and $\|\EE_h u-\EE_h v\|_\ka\leq e^{\al h}\|u-v\|_\ka$ for all $h\geq 0$ and $u,v\in \UC_\ka$, it follows that \eqref{contmap} is continuous.
\end{proof}

Let $u\in \UC_\ka$. In the following, we consider the limit of $\EE_\pi u$ when the mesh size
\[
 |\pi|_\infty := \max_{j=1,\dots,m} (t_j-t_{j-1})
\]
of the partition $\pi=\{t_0,t_1,\dots,t_m\}\in P$ with $0=t_0< t_1< \ldots < t_m$ tends to zero. For the sake of completeness, we define $|\{0\}|_\infty:=0$. The following lemma shows that $\SS(t)u$ can be obtained by a pointwise monotone approximation with finite partitions letting the mesh size tend to zero.

\begin{proposition}\label{monlimit}
 Assume that the map \eqref{ass.approx} is continuous for all $u\in \UC_\ka$. Let $t\geq 0$ and $(\pi_n)_{n\in \N}\subset P_t$ with $\pi_n\subset \pi_{n+1}$ for all $n\in \N$ and $|\pi_n|_\infty\searrow 0$ as $n\to \infty$. Then, for all $u\in \UC_\ka$,
 \[
  \EE_{\pi_n}u\nearrow \SS(t)u \quad \text{as }n\to \infty.
 \]
 In particular,
 \[
  \SS(t)u=\sup_{n\in \N} \EE_{\frac{t}{n}}^n u=\lim_{n\to \infty}\EE_{2^{-n}t}^{2^n}u\quad \text{for all }u\in \UC_\ka,
 \]
 where the supremum and the limit are to be understood in a pointwise sense.
\end{proposition}

\begin{proof}
 For $t=0$ the statement is trivial. Therefore, assume that $t>0$, and let
 \[
  \big(\EE_\infty u\big)(x):=\sup_{n\in \N} \big(\EE_{\pi_n}u\big)(x)\quad \text{for }u\in \UC_\ka\text{ and }x\in M.
 \]
  As in the proof of Lemma \ref{seq}, the map $\EE_\infty\colon \UC_\ka\to \UC_\ka$ is well-defined. Let $u\in \UC_\ka$. Since $\pi_n\subset \pi_{n+1}$ for all $n\in \N$, it follows that $\EE_{\pi_n}u\nearrow \EE_\infty u$ as $n\to \infty$. Since $(\pi_n)_{n\in \N}\subset P_t$, we obtain that
 \[
  \EE_\infty u \leq \SS(t)u.
 \]
 Let $\pi=\{t_0,t_1,\ldots, t_m\}\in P_t$ with $m\in \N$ and $0=t_0<t_1<\ldots <t_m=t$. Since $|\pi_n|_\infty\searrow 0$ as $n\to \infty$, we may w.l.o.g.~assume that $\#\pi_n\geq m+1$ for all $n\in \N$. Let $0=t_0^n<t_1^n<\ldots <t_m^n=t$ for all $n\in \N$ with $\pi_n':=\{t_0^n,t_1^n,\ldots, t_m^n\}\subset \pi_n$ and $t_i^n\to t_i$ as $n\to \infty$ for all $i\in \{1,\ldots, m\}$. Then, by Lemma \ref{strongcontJ},
 \[
  \|\EE_\pi u-\EE_{\pi_n'}u\|_\ka\to 0 \quad \text{as }n\to \infty.
 \]
 Therefore,
 \[
  \EE_\infty u-\EE_\pi u\geq \EE_{\pi_n}u-\EE_\pi u\geq \EE_{\pi_n'}u-\EE_\pi u\to 0\quad \text{as }n\to \infty,
 \]
 showing that $\EE_\infty u\geq \EE_\pi u$. Taking the supremum over all $\pi\in P_t$, we obtian that $\EE_\infty u=\SS(t) u$.
 
 Now, let $\pi_n:=\big\{\tfrac{kt}{2^n}\, \big|\, k\in \{0,\ldots ,2^n\}\big\}$ for all $n\in \N$. Then,
 \[
  \SS(t)u=\lim_{m\to \infty}\EE_{\pi_m}u=\lim_{m\to \infty}\EE_{2^{-m}t}^{2^m}u\leq \sup_{n\in \N} \EE_{\frac{t}{n}}^n u\leq \SS(t)u,
 \]
 where we used the basic fact that $n=2^m\in \N$ for all $m\in \N$.
\end{proof}

\section{Strong continuity}\label{sec:strongcont}

Let $\SS$ be the Feller semigroup from the previous section, i.e. the semigroup envelope of the family $(S_\la)_{\la\in \La}$. The aim of this section is to give conditions that ensure the strong continuity of the semigroup envelope $\SS$.

\begin{remark}\label{reminvariant}
 Let $D\subset \UC_\ka$ be the set of \textit{all} $u\in \UC_\ka$, for which the map
 \[
  [0,\infty)\to \UC_\ka,\quad  t\mapsto \SS(t)u
 \]
 is continuous. Then, by the semigroup property \eqref{dpp1}, $$[0,\infty)\to \UC_\ka,\quad   s\mapsto \SS(s)\SS(t)u= \SS(s+t)u$$ is continuous for all $u\in D$. Therefore, the set $D$ is invariant under the semigroup $\SS$, i.e. $\SS(t)u\in D$ for all $u\in D$ and all $t\geq 0$.
\end{remark}

\begin{lemma}\label{equivscont}
 Let $u\in \UC_\ka$. Then, the following statements are equivalent:
 \begin{enumerate}
  \item[(i)] $\lim_{h\to 0}\|\SS(h)u-u\|_\ka=0$.
  \item[(ii)] The map $[0,\infty)\to \UC_\ka,\;  t\mapsto \SS(t)u$ is continuous.
 \end{enumerate}

\end{lemma}

\begin{proof}
 Clearly, (ii) implies (i). Therefore, assume that $\lim_{h\to 0}\|\SS(h)u-u\|_\ka=0$. Let $t\geq 0$ and $\ep>0$. W.l.o.g. we may assume that in (A2) we have $\al\geq 0$. By assumption, there exists some $\de>0$ such that $\|\SS(h)u-u\|_\ka<e^{-\al t}\ep$ for all $h\in [0,\de)$. Now, let $s\geq 0$ with $|t-s|<\de$. Then, for $\tau:=s\wedge t$,
 \begin{align*}
  \|\SS(t)u-\SS(s)u\|_\ka&=\big\|\SS(\tau)\big(\SS(|t-s|)u\big)-\SS(\tau)u\big\|_\ka\\
  &\leq e^{\al \tau}\big\|\SS\big(|t-s|\big)u-u\big\|_\ka<\ep,
 \end{align*}
 where we used the Lipschitz continuity of $\SS(\tau)$ with Lipschitz constant $e^{\al \tau}$.
\end{proof}

\begin{remark}\label{remDclosed}
 Let $D\subset \UC_\ka$ arbitrary, and assume that $\SS$ is strongly continuous on $D$. Then, $\SS$ is also strongly continuous on the closure $\overline D$ of $D$. In order to see this, let $u\in \UC_\ka$ and $(u_n)_{n\in \N}\subset D$ with $\|u_n-u\|_\ka\to 0$ as $n\to \infty$. W.l.o.g. we may assume that $\al\geq 0$. Let $\ep>0$. Then, there exists some $n_0\in \N$ such that $\|u_{n_0}-u\|_\ka \leq \tfrac{\ep}{3}e^{-\al}$. Since $u_{n_0}\in D$, there exists some $\de\in (0,1]$ such that $\|\SS(h)u_{n_0}-u_{n_0}\|_\ka <\tfrac{\ep}{3}$ for all $h\in [0,\de)$. Hence, for $h\in [0,\de)$, it follows that
 \[
  \|\SS(h)u-u\|_\ka \leq \frac{2\ep}{3} + \|\SS(h)u_{n_0}-u_{n_0}\|_\ka<\ep.
 \]
 Now, the previous lemma implies that $[0,\infty)\to \UC_\ka,\; t\mapsto \SS(t)u$ is continuous.
\end{remark}

We start with the first result ensuring the strong continuity of the semigroup envelope $\SS$.

\begin{proposition}\label{critstrongcont1}
 Assume that, for every $\de>0$, there exists a family of functions $(\varphi_x)_{x\in M}\subset \UC_\ka$ satisfying the following:
 \begin{enumerate}
  \item[(i)] $0\leq \varphi_x(y)\leq 1$ for all $y\in M$, $\varphi_x(x)=0$, $\varphi_x(y)=1$ for all $y\in M$ with $d(x,y)\geq \de$,
  \item[(ii)] $\sup_{x\in M} \ka(x)\big[\big(\SS(h)\varphi_x\big)(x)\big] \to 0$ as $h\searrow 0$.
  \end{enumerate}
 Then, the semigroup $\SS$ is strongly continuous.
\end{proposition}

\begin{proof}
 Let $u\in \Lipb\sm\{0\}$ and $\ep>0$. Then, since $\ka$ is bounded, there exists some $\de>0$ such that
 \[
   \ka (y)|u(y)-u(x)|\leq \tfrac{\ep}{2e^{\al}}\quad \text{for all }x,y\in M\text{ with }d(x,y)<\de.
 \]
 By assumption, there exists a family $(\varphi_x)_{x\in M}\subset \UC_\ka$ with $0\leq \varphi_x(y)\leq 1$ for all $y\in M$, $\varphi_x(x)=0$, $\varphi_x(y)=1$ for all $y\in M$ with $d(x,y)\geq \de$, and some $h_0\in (0,1]$ such that
 \[
  \sup_{x\in M} \ka(x)\big[\big(\SS(h)\varphi_x\big)(x)\big]< \frac{\ep}{4\|u\|_\infty}\quad\text{for all }h\in [0,h_0). 
 \]
 For all $h\in (0,1]$ and $x\in M$,
 \begin{align*}
  \Big\|\SS(h)\big((1-\varphi_x)|u-u(x)|\big)\Big\|_\ka&\leq e^\al\big\|(1-\varphi_x)|u-u(x)|\big\|_\ka\\
  &\leq e^\al\sup_{\substack{y\in M\\d(x,y)\leq \de}} \ka(y)|u(y)-u(x)|\leq\frac\ep 2.
 \end{align*}
Hence, for all $h\in [0,h_0)$ and $x\in M$, since $\SS(h)1=1$,
 \begin{align*}
  \ka (x)\big|\big(\SS(h)u\big)(x)-u(x)\big|&=\ka(x)\big|\big(\SS(h)(u-u(x))\big)(x)\big|\\
  &\leq \ka (x)\big(\SS(h)|u-u(x)|\big)(x)\\
  &\leq \Big\|\SS(h)\big((1-\varphi_x)|u-u(x)|\big)\Big\|_\ka\\
  &\quad +\ka (x)\big(\SS(h)(\varphi_x|u-u(x)|)\big)(x)\\
  &\leq \frac{\ep}{2}+2\ka (x)\|u\|_\infty\big(\SS(h)\varphi_x\big)(x)< \ep.
 \end{align*}
 This shows that $\|\SS(h)u-u\big\|_\ka <\ep$ for all $h\in [0,h_0)$, and therefore $\SS$ is strongly continuous on $\Lipb$. Since $\Lipb$ is, by definition, dense in $\UC_\ka$, Remark \ref{remDclosed} implies that $\SS$ is strongly continuous.
\end{proof}

 The function $\ph_x$, for $x\in M$, in the previous proposition plays the role of a cut-off function. Proposition \ref{critstrongcont1} is a generalisation of the well-known fact that transition semigroups of L\'evy processes are strongly continuous, where the strong continuity is intimately related to the convergence in law of the process. Note that, for transition semigroups of L\'evy processes, the translation invariance together with the convergence in law ensures that the assumptions of Proposition \ref{critstrongcont1} are satisfied.

We denote by $D_\La\subset \UC_\ka$ the linear space of all $u\in \UC_\ka$ for which there exist $L_u\geq 0$ and $h_u>0$ such that
\[
 \sup_{\la\in \La}\|S_\la (h)u-u\|_\ka \leq L_u h\quad \text{for all }h\in [0,h_u).
\]

\begin{proposition}\label{scontin}
 The semigroup $\SS$ is strongly continuous on $\overline{D_\La}$. In particular, $\SS$ is strongly continuous if $D_\La$ is dense in $\UC_\ka$.
\end{proposition}

\begin{proof}
 Let $u\in D_\La$ and $0\leq h_1<h_2$ with $h_2-h_1<h_u$. Then,
 \[
  \big(S_{\la_0}(h_1)u\big)(x)-\big(\EE_{h_2}u\big)(x)\leq \big(S_{\la_0}(h_1)u\big)(x)-\big(S_{\la_0}(h_2)u\big)(x)
 \]
 for all $x\in M$ and $\la_0\in \La$. Taking the supremum over $\la_0\in \La$, it follows that
 \[
  \big(\EE_{h_1}u\big)(x)-\big(\EE_{h_2}u\big)(x)\leq \sup_{\la\in \La}\big|\big(S_\la(h_1)u\big)(x)-\big(S_\la(h_2)u\big)(x)\big|
 \]
 for all $x\in M$. By a symmetry argument, multiplying by $\kappa(x)$ and taking the supremum over all $x\in M$, we obtain that $\|\EE_{h_1}u-\EE_{h_2}u\|_\ka\leq \sup_{\la\in \La}\|S_\la(h_1)u-S_\la(h_2)u\|_\ka$. Moreover,
 \[
  \|S_\la(h_1)u-S_\la(h_2)u\|_\ka\leq e^{\al h_1}\|S_\la(h_2-h_1)u-u\|_\ka\leq L_ue^{\al h_1} (h_2-h_1).
 \]
 Taking the supremum over all $\la\in \La$, we obtain that
 \begin{equation}\label{strongcont:1}
  \|\EE_{h_1}u -\EE_{h_2}u\|_\ka\leq L_u e^{\al h_1} (h_2-h_1).
 \end{equation}
 Next, we show that
  \begin{equation}\label{strongcont:2}
  \|\EE_\pi u -u\|_\ka\leq L_ue^{\al \max \pi} \max \pi
 \end{equation}
 for all $\pi\in P$ with $\max\pi\in [0,h_u)$ by an induction on $\# \pi\in \N$. First, let $\pi \in P$ with $\#\pi =1$, i.e. $\pi=\{0\}$. Then,
 \[
  \|\EE_\pi u -u\|_\ka=\|\EE_{\{0\}}u-u\|_\ka=0=L_u e^{\al \max \pi} \max \pi.
 \]
 Now, let $m\in \N$, and assume that \eqref{strongcont:2} holds for all $\pi\in P$ with $\max\pi\in [0,h_u)$ and $\#\pi=m$. Let $\pi\in P$ with $\#\pi =m+1$ and $t_m:=\max \pi\in [0,h_u)$. Then, $\pi':=\pi\setminus \{t_m\}\in P$ with $\#\pi'=m$ and $t_{m-1}:=\max \pi'\in [0,t_m)$. Therefore, by induction hypothesis and \eqref{strongcont:1}, it follows that
 \begin{align*}
  \|\EE_\pi u-u\|_\ka &\leq \|\EE_\pi u-\EE_{\pi'}u\|_\ka+\|\EE_{\pi'}u-u\|_\ka\\
  &=\|\EE_{\pi'}\EE_{t_m-t_{m-1}}u-\EE_{\pi'}u\|_\ka+\|\EE_{\pi'}u-u\|_\ka\\
  &\leq e^{\al t_{m-1}}\|\EE_{t_m-t_{m-1}}u-u\|_\ka+\|\EE_{\pi'}u-u\|_\ka\\
  &\leq L_u e^{\al t_{m-1}}(t_m-t_{m-1})+L_ue^{\al t_{m-1}}t_{m-1}\\
  &= L_u e^{\al t_{m-1}}t_m\leq L_ue^{\al \max \pi}\max \pi.
 \end{align*}
 By definition of the semigroup $\SS$, we thus obtain that
 \[
  \|\SS(h)u-u\|_\ka \leq L_u e^{\al h}h \to 0\quad \text{as }h\to 0.
 \]
\end{proof}

The following proposition is somewhat similar to Proposition \ref{critstrongcont1}. Note that (ii) in Proposition \ref{critstrongcont1} is a condition related to the semigroup envelope $\SS$, and its verification is typically nontrivial. The following proposition replaces condition (ii) in Proposition \ref{critstrongcont1} by a smoothness condition on the cut-off functions $(\ph_x)_{x\in M}$, where smoothness is given in terms of the family of generators $(A_\lambda)_{\lambda\in \Lambda}$.
\begin{proposition}\label{critstrongcont2}
 Assume that for every $\de>0$ there exists a family of functions $(\varphi_x)_{x\in M}\subset \UC_\ka$ satisfying the following:
 \begin{enumerate}
  \item[(i)] $0\leq \varphi_x(y)\leq 1$ for all $y\in M$, $\varphi_x(x)=0$, and $\varphi_x(y)=1$ for all $y\in M$ with $d(x,y)\geq \de$,
  \item[(ii')] There exist $L\geq 0$ and $h_0>0$ such that, for all $h\in [0,h_0)$ and $x\in M$,
 \[
  \sup_{\la\in \La}\|S_\la (h)\varphi_x-\varphi_x\|_\ka \leq L h.
 \]
 \end{enumerate}
 Then, $\SS$ is strongly continuous.
\end{proposition}

\begin{proof}
 By assumption, the family $(\varphi_x)_{x\in M}$ satisfies condition (i) from Proposition \ref{critstrongcont1}. We now verify that (ii') implies condition (ii) from Proposition \ref{critstrongcont1}. Observe that
 \[
  \big(\EE_h \varphi_x\big)(y)\leq \varphi_x(y)+ \big| \big(\EE_h \varphi_x\big)(y)-\varphi_x(y)\big|
 \]
 for all $h\in [0,h_0)$ and $x,y\in M$. W.l.o.g. we assume that $\al\geq 0$ in (A2). Then, by \eqref{strongcont:1}, we obtain that
 \begin{align*}
  \big( \EE_\pi\EE_h \varphi_x\big)(x)&\leq \big(\EE_\pi\varphi_x\big)(x)+\big(\EE_\pi \big| \EE_h \varphi_x-\varphi_x\big|\big)(x)\\
  &\leq \big(\EE_\pi\varphi_x\big)(x)+\frac{e^{\al h_0}}{\ka(x)}\| \EE_h \varphi_x-\varphi_x\|_\ka\\
  &\leq \big(\EE_\pi\varphi_x\big)(x)+\frac{Le^{\al h_0}h}{\ka(x)}
 \end{align*}
 for all $\pi\in P$ with $\max\pi\in [0,h_0)$ and $h\in [0,h_0)$. Inductively, it follows that
 \[
   \big(\EE_\pi \varphi_x\big)(x)\leq \varphi_x(x)+\frac{Le^{\al h_0}\max \pi}{\ka(x)}=\frac{Le^{\al h_0}\max \pi}{\ka(x)}
 \]
 for all $\pi \in P$ with $\max\pi\in [0,h_0)$. Taking the supremum over all $\pi \in P_h$ for $h\in [0,h_0)$ yields that
 \[
  \sup_{x\in M}\ka(x)\big[\big(\SS(h)\varphi_x\big)(x)\big]\leq Le^{\al h_0}h\to 0\quad\text{as }h\searrow0.
 \]
 Therefore, condition (ii) from Proposition \ref{critstrongcont1} is satisfied and the strong continuity of $\SS$ follows.
\end{proof}

\section{Related HJB equation and viscosity solutions}\label{sec:viscosity}

Let $\la\in \La$. Then, we denote by $D_\la\subset \UC_\ka$ the space of all $u\in \UC_\ka$ such that the map $[0,\infty)\to \UC_\ka,\; t\mapsto S_\la(t)u$ is continuous. Further, let $D(A_\la)$ denote the space of all $u\in \UC_\ka$ for which
\[
 A_\la u:=\lim_{h\searrow 0}\frac{S_\la(h)u-u}{h}\in \UC_\ka
\]
exists w.r.t.\ $\|\cdot\|_\kappa$. Note that, by definition, $D(A_\la)\subset D_\la$. Let $u\in \bigcap_{\la\in \La}D(A_\la)$ with $C_u:=\sup_{\la\in \La}\|A_\la u\|_\ka<\infty$. Then, it follows that (see e.g. \cite[Lemma II.1.3]{MR1721989})
 \[
  \|S_\la(h)u-u\|_\ka\leq \int_0^h\|S_\la(s) A_\la u\|_\ka \, {\rm d}s\leq C_ue^{\al h}h\quad \text{for all }\la\in \La.
 \]
 This shows that $u\in D_\La$. Moreover, since $\sup_{\la\in \La} \|A_\la u\|_\ka<\infty$, it follows that
 \[
  \big(\AA u\big)(x):=\sup_{\la \in \La} \big(A_\la u\big)(x)
 \]
 is well-defined for all $x\in M$.

\begin{lemma}\label{generatorprep1}
 Let $u\in \bigcap_{\la\in \La}D(A_\la)$ with
 \[
  \sup_{\la\in \La}\|A_\la u\|_\ka<\infty \quad \text{and}\quad \sup_{\la\in \La}\|S_\la(h)A_\la u-A_\la u\|_\ka\to 0\quad \text{as }h\to 0.
 \]
 Then, $\lim_{h\searrow 0}\big\|\frac{\EE_h u-u}{h}- \AA u\big\|_\ka=0$. In particular, $\AA u\in \UC_\ka$.
\end{lemma}

\begin{proof}
 Let $\ep>0$. Then, by assumption, there exists some $h_0>0$ such that
 \[
  \sup_{\la\in \La} \|S_\la(s) A_\la u-A_\la u\|_\ka\leq \ep\quad\text{for all }s\in [0,h_0].
 \]
 Hence, for all $h\in (0,h_0]$, it follows that
 \begin{align*}
  \bigg\|\frac{\EE_hu-u}{h}- \AA u\bigg\|_\ka&\leq \sup_{\la \in \La}\bigg\|\frac{S_\la(h)u-u}{h}- A_\la u\bigg\|_\ka\\
  &=\sup_{\la\in \La}\frac{1}{h} \bigg\|\int_0^hS_\la(s)A_\la u-A_\la u\, {\rm d}s\bigg\|_\ka\\
  &\leq \sup_{\la\in \La}\frac{1}{h} \int_0^h\|S_\la(s)A_\la u-A_\la u\|_\ka\, {\rm d}s\leq \ep.
 \end{align*}
\end{proof}

\begin{proposition}\label{generator1}
  Let $u\in \bigcap_{\la\in \La}D(A_\la)$ with
 \[
  \sup_{\la\in \La}\|A_\la u\|_\ka<\infty \quad \text{and}\quad \sup_{\la\in \La}\|S_\la(h)A_\la u-A_\la u\|_\ka\to 0\quad \text{as }h\to 0.
 \]
 Then, $\AA u\in \UC_\ka$ and the following statements are equivalent:
 \begin{enumerate}
  \item[(i)] The map $[0,\infty)\to \UC_\ka,\; t\mapsto \SS(t)\AA u$ is continuous,
  \item[(ii)] $\lim_{h\searrow 0}\big\|\tfrac{\SS(h)u-u}{h}- \AA u\big\|_\ka=0$, i.e. $\AA u=\lim_{h\searrow 0} \frac{\SS(h)u-u}{h}$, where the limit is w.r.t. $\|\cdot\|_\kappa$.
 \end{enumerate}
\end{proposition}

\begin{proof}
 By Lemma \ref{generatorprep1}, we already know that $\AA u\in \UC_\ka$. Let $D$ denote the set of all $v\in \UC_\ka$, for which the map $[0,\infty)\to \UC_\ka,\; t\mapsto \SS(t)v$ is continuous. Our assumptions imply that $u\in D_\La$. Therefore, by Proposition \ref{scontin}, $u\in D$ and, by Remark \ref{reminvariant}, $\SS(h)u\in D$ for all $h\geq 0$. Hence, by Remark \ref{remDclosed}, statement (ii) implies (i). By Lemma \ref{generatorprep1},
 \[
  \AA u-\frac{\SS(h)u-u}{h}\leq \AA u -\frac{\EE_hu-u}{h}\to 0,\quad \text{as }h\searrow 0.
 \]
 Assuming that the map $[0,\infty)\to \UC_\ka,\; t\mapsto \SS(t)\AA u$ is continuous, it follows that
 \[
 \bigg\|\frac{1}{h}\int_0^h \SS(s)\AA u \, {\rm d}s-\AA u\bigg\|_\ka\to 0,\quad \text{as }h\searrow 0.
 \]
 Hence, it is sufficient to show that
 \begin{equation}\label{eq:intsemi1}
  \SS(t)u-u\leq \int_0^t \SS(s)\AA u \, {\rm d}s\quad\text{for all }t\geq 0.
 \end{equation}
 Let $t\geq 0$ and $h>0$. Then,
 \begin{equation}\label{intsemieq}
  \EE_h u-u=\sup_{\la \in \La} \int_0^h S_\la(s)A_\la u\, {\rm d} s \leq \int_0^h\SS(s)\AA u\,{\rm d} s=\int_t^{t+h}\SS(s-t)\AA u\, {\rm d}s.
 \end{equation}
 Next, we prove that
 \[
  \EE_\pi u-u\leq \int_0^{\max \pi} \SS(s)\AA u \, {\rm d}s \quad\text{for all }\pi\in P
 \]
  by an induction on $m=\# \pi$. If $m=1$, i.e.~if $\pi=\{0\}$, the statement is trivial. Hence, assume that
 \[
  \EE_{\pi'} u-u\leq \int_0^{\max \pi'} \SS(s)\AA u \, {\rm d}s
 \]
 for all $\pi'\in P$ with $\#\pi'=m$ for some $m\in \N$. Let $\pi=\{t_0,t_1,\ldots, t_m\}\in P$ with $0=t_0< t_1< \ldots < t_m$ and $\pi':=\pi\sm \{t_m\}$. Then, it follows from \eqref{intsemieq} that
 \begin{align*}
  \EE_\pi u -\EE_{\pi'}u&\leq \SS(t_{m-1})\big(\EE_{t_m-t_{m-1}}u-u\big)\leq \SS(t_{m-1})\bigg(\int_{t_{m-1}}^{t_m} \SS(s-t_{m-1})\AA u\, {\rm d}s\bigg)\\
  &\leq\int_{t_{m-1}}^{t_m}\SS(s)\AA u\, {\rm d}s,
 \end{align*}
 where the last inequality follows from Jensen's inequality. By induction hypothesis, we thus obtain that
 \begin{align*}
 \EE_\pi u -u&= \big(\EE_\pi u -\EE_{\pi'}u\big)+\big(\EE_{\pi'} u-u\big)\leq \int_{t_{m-1}}^{t_m}\SS(s)\AA u\, {\rm d}s+\int_{0}^{t_{m-1}}\SS(s)\AA u\, {\rm d}s\\
 &=\int_0^{\max \pi}\SS(s)\AA u\, {\rm d}s.
 \end{align*}
 In particular, $\EE_\pi u-u\leq \int_0^t\SS(s)\AA u\, {\rm d}s$ for every $\pi\in P_t$. Taking the supremum over all $\pi\in P_t$ yields the assertion.
\end{proof}

We now introduce the class of test functions, which will be used for the definition of a viscosity solution. Let $$\DD:=\bigg\{u\in \bigcap_{\la\in \La}D(A_\la)\, \bigg|\, \sup_{\la\in \La}\|A_\la u\|_\ka<\infty \text{ and } \lim_{h\searrow 0}\bigg\|\frac{\SS(h)u-u}{h}-\AA u\bigg\|_\ka= 0 \bigg\}.$$
In the sequel, we are interested in viscosity solutions to the differential equation
 \begin{equation}\label{eq:PDE}
  u'(t)= \AA u(t), \quad \text{for }t> 0,
 \end{equation}
where we use the following notion of a viscosity solution.

\begin{definition}\label{def.viscosity}
 We say that $u\colon [0,\infty)\to \UC_\ka$ is a \textit{viscosity subsolution} to \eqref{eq:PDE} if $u$ is continuous, and for every $t>0$, $x\in M$, and every differentiable function $\psi\colon (0,\infty)\to \UC_\ka$ with $\psi(t)\in \DD$, $\big(\psi(t)\big)(x)=\big(u(t)\big)(x)$ and $\psi(s)\geq u(s)$ for all $s>0$,
 \[
  \big(\psi'(t)\big)(x)\leq \big(\AA \psi(t)\big)(x).
 \]
 Analogously, $u$ is called a \textit{viscosity supersolution} to \eqref{eq:PDE} if $u\colon [0,\infty)\to \UC_\ka$ is continuous, and for every $t>0$, $x\in M$, and every differentiable function $\psi\colon (0,\infty)\to \UC_\ka$ with $\psi(t)\in \DD$, $\big(\psi(t)\big)(x)=\big(u(t)\big)(x)$ and $\psi(s)\leq u(s)$ for all $s>0$,
 \[
  \big(\psi'(t)\big)(x)\geq \big(\AA \psi(t)\big)(x).
 \]
 We say that $u$ is a \textit{viscosity solution} to \eqref{eq:PDE} if $u$ is a viscosity subsolution and a viscosity supersolution.
\end{definition}

\begin{remark}\label{rem.uniqueness}
 In general it is not clear how rich the class of test functions for a viscosity solution from the previous definition is. However, in the examples in Section \ref{sec:ex}, we will see that, in most cases, where $M$ is a Banach space, $\Lipb^k\subset \DD$ with $k\in \{0,1,2\}$, where $\Lipb^k$ denotes the set of all $k$-times (Fr\'echet) differentiable functions $M\to \R$ with bounded and Lipschitz continuous derivatives. For a function $\psi\colon (0,\infty)\times M\to \R$, which is differentiable w.r.t.~$t$ and $\partial_t\psi\colon (0,\infty)\times M\to \R$ uniformly w.r.t.~$x$ Lipschitz continuous in $t$ with Lipschitz constant $L\geq 0$, it follows that
\[
 \sup_{x\in M}\bigg|\frac{\psi(t+h,x)-\psi(t,x)}{h}-\partial_t\psi(t,x)\bigg| \leq Lh \to 0 \quad \text{as }h\searrow 0
\]
for all $t>0$. Hence, if $\Lipb^k\subset \DD$ for some $k\in \N_0$, every $\psi\in \Lipb^{1,k}\big((0,\infty)\times M\big)$ is differentiable as a map $(0,\infty)\to \UC_\ka$ and satisfies $\psi(t)\in \DD$ for all $t>0$. In most applications, the class $\Lipb^{1,k}\big((0,\infty)\times M\big)$ of test functions is sufficiently large in order to obtain uniqueness of a viscosity solution. For more details concerning our notion of a viscosity solution and the uniqueness of solutions, we refer to Section \ref{sec:uniqueness}. 
\end{remark}

We conclude this section with the following main theorem.

\begin{theorem}\label{viscosity1}
 Assume that the semigroup $\SS$ is strongly continuous. Then, for every $u_0\in \UC_\ka$, the function $u\colon [0,\infty)\to \UC_\ka, \;t\mapsto \SS(t)u_0$
 is a viscosity solution to the abstract initial value problem
\begin{eqnarray*}
  u'(t)&=&\AA u(t), \quad \text{for }t> 0,\\
  u(0)&=&u_0.
 \end{eqnarray*}
\end{theorem}

\begin{proof}
 Fix $t>0$ and $x\in M$. We first show that $u$ is a viscosity subsolution. Let $\psi\colon (0,\infty)\to \UC_\ka$ differentiable with $\psi(t)\in \DD$, $\big(\psi(t)\big)(x)=\big(u(t)\big)(x)$ and $\psi(s)\geq u(s)$ for all $s>0$. Then, for every $h\in (0,t)$, it follows from Equation \eqref{dpp1} that
 \begin{align*}
  0&=\frac{\SS(h)\SS(t-h)u_0-\SS(t)u_0}{h}=\frac{\SS(h)u(t-h)-u(t)}{h}\\
  &\leq \frac{\SS(h)\psi(t-h)-u(t)}{h} \leq \frac{\SS(h)\big(\psi(t-h)-\psi(t)\big)+\SS(h)\psi(t)-u(t)}{h}\\
  &= \SS(h)\bigg(\frac{\psi(t-h)-\psi(t)}{h}\bigg)+\frac{\SS(h)\psi(t)-\psi(t)}{h}+\frac{\psi(t)-u(t)}{h}.
 \end{align*}
 Moreover,
 \begin{align*}
  &\bigg\|\SS(h)\bigg(\frac{\psi(t-h)-\psi(t)}{h}\bigg)+\psi'(t)\bigg\|_\ka\to 0\quad \text{and}\\
  &\bigg\|\frac{\SS(h)\psi(t)-\psi(t)}{h}-\AA \psi(t)\bigg\|_\ka\to 0.
 \end{align*}
 as $h\searrow 0$. Since $\big(u(t)\big)(x)=\big(\psi(t)\big)(x)$, it follows that
 \[
  0\leq -\big(\psi'(t)\big)(x)+\big(\AA\psi(t)\big)(x).
 \]
 In order to show that $u$ is a viscosity supersolution, let $\psi\colon (0,\infty)\to \UC_\ka$ differentiable with $\psi(t)\in \DD$, $\big(\psi(t)\big)(x)=\big(u(t)\big)(x)$ and $\psi(s)\leq u(s)$ for all $s>0$. By Equation \eqref{dpp1}, for all $h>0$ with $0<h<t$, we obtain that
 \begin{align*}
  0&=\frac{\SS(t)u_0-\SS(h)\SS(t-h)u_0}{h}\\
  &=\frac{u(t)-\SS(h)u(t-h)}{h}\leq \frac{u(t)-\SS(h)\psi(t-h)}{h}\\
  &= \frac{u(t)-\psi(t)}{h}+\frac{\psi(t)-\SS(h)\psi(t)}{h}+\frac{\SS(h)\psi(t)-\SS(h)\psi(t-h)}{h}\\
  &\leq \frac{u(t)-\psi(t)}{h}+\frac{\psi(t)-\SS(h)\psi(t)}{h}+\SS(h)\bigg(\frac{\psi(t)-\psi(t-h)}{h}\bigg).
 \end{align*}
 Furthermore,
 \begin{align*}
  &\bigg\|\SS(h)\bigg(\frac{\psi(t)-\psi(t-h)}{h}\bigg)- \psi'(t)\bigg\|_\ka\to 0\quad \text{and}\\
  &\bigg\|\frac{\psi(t)-\SS(h)\psi(t)}{h} +\AA\psi(t)\bigg\|_\ka\to 0.
 \end{align*}
  Since $\big(u(t)\big)(x)=\big(\psi(t)\big)(x)$, we obtain that $0\leq -\big(\AA\psi(t)\big)(x)+\big(\psi'(t)\big)(x)$.
\end{proof}

\section{Stochastic representation}\label{sec:expec}

In this section, we derive a stochastic representation for the semigroup envelope $\SS$ using sublinear expectations. Such stochastic representations are of fundamental interest in various fields and, in particular, in the field of robust finance. The prime example for a sublinear expectation arising from a semigroup envelope for a particular family of semigroups is the $G$-expectation, cf.~Denis et al. \cite{MR2754968} and Peng~\cite{PengG},\cite{MR2474349}, and the corresponding Markov process, the $G$-Brownian Motion, is the analogue of a Brownian Motion in the presence of volatility uncertainty. More general forms of stochastic processes arising from semigroups are given by the class of so-called $G$-L\'evy processes, cf.\ Hu and Peng \cite{PengHu}, Neufeld and Nutz \cite{NutzNeuf}, and Denk et al.\ \cite{dkn}. In this section, we provide a similar representation for $\SS$ under an additional continuity assumption. We point out that our setup covers the aforementioned existing approaches. We start with a short introduction to the theory of nonlinear expectations. For a measurable space $(\Om,\FF)$, we denote the space of all bounded $\FF$-measurable functions (random variables) $\Om\to \R$ by $\LL^\infty(\Om,\FF)$. For two bounded random variables $X,Y\in \LL^\infty(\Om,\FF)$ we write $X\leq Y$ if $X(\om)\leq Y(\om)$ for all $\om\in \Om$. For a constant $\al\in \R$, we do not distinguish between $\al$ and the constant function taking the value $\al$.

\begin{definition}
 Let $(\Om,\FF)$ be a measurable space. A functional $\EE\colon \LL^\infty(\Om,\FF)\to\R$ is called a \textit{sublinear expectation} if for all $X,Y\in \LL^\infty(\Om,\FF)$ and $\la>0$
 \begin{enumerate}
  \item[(i)] $\EE(X)\leq \EE(Y)$ if $X\leq Y$,
  \item[(ii)] $\EE(\al)=\al$ for all $\al\in \R$,
  \item[(iii)] $\EE(X+Y)\leq \EE(X)+\EE(Y)$ and $\EE(\la X)=\la \EE(X)$.
 \end{enumerate}
  We say that $(\Om,\FF,\EE)$ is a \textit{sublinear expectation space} if there exists a set of probability measures $\PP$ on $(\Om,\FF)$ such that
  \[
   \EE(X)=\sup_{\P\in \PP} \E_\P(X)\quad \text{for all }X\in \LL^\infty(\Om,\FF),
  \]
  where $\E_\P(\cdot)$ denotes the expectation w.r.t. to the probability measure $\P$.
\end{definition}

\begin{definition}
 Let $L\subset \UC_\ka$ be a linear space. We say that $\SS$ is \textit{continuous from above} on $L$ if $\SS(t)u_n\searrow 0$ for all $t\geq 0$ and all $(u_n)_{n\in \N}\subset L$ with $u_n\searrow 0$ as $n\to \infty$.
\end{definition}

\begin{remark}\label{rem.contabove}\
\begin{enumerate}
  \item[a)] Assume that $M$ is compact. Then, by Dini's lemma, $\SS$ is continuous from above on $\UC_\kappa=\UC_{\rm b}$.
  \item[b)] Assume that $M$ satisfies the Heine-Borel property, i.e. every closed and bounded subset of $M$ is compact, and that $\ka\in \Con_0$. Then, $\UC_\ka=\{u\in \Con\, |\, \ka u\in \Con_0\}$, where $\Con_0$ denotes the closure of the space $\Con_c$ of all continuous functions with compact support w.r.t.\ $\|\cdot\|_\infty$. In fact, let $u\in \UC_\kappa$. Then, there exists a sequence $(u_n)_{n\in \N}\subset \Lipb$ with $\|u-u_n\|_\ka\to 0$ as $n\to \infty$. Since $\ka\in \Con_0$, it follows that $v_n:=\ka u_n\in \Con_0$ for all $n\in \N$. Since $\Con_0$ endowed with $\|\cdot\|_\infty$ is a Banach space and 
  \[
   \|\ka u-v_n\|_\infty=\|u-u_n\|_\kappa\to 0\quad \text{as }n\to \infty,
  \]
  we find that $\ka u\in \Con_0$. Now, assume that $\ka u\in \Con_0$. Then, there exists a sequence $(v_n)_{n\in \N}\subset \Con_c$ with $\|\ka u-v_n\|_\infty\to 0$. Defining $u_n:=\frac{v_n}{\ka}$ for $n\in \N$, we see that $u_n\in \Con_0\subset \UC_{\rm b}$. Since $\UC_{\rm b}\subset \UC_\kappa$ and
  \[
   \|u-u_n\|_\ka=\|\ka u-v_n\|_\infty\to 0,
  \]
  it follows that $u\in \UC_\kappa$. We have therefore established the equality $\UC_\ka=\{u\in \Con\, |\, \ka u\in \Con_0\}$. Let $(u_n)_{n\in \N}\subset \UC_\kappa$ with $u_n\searrow 0$ as $n\to \infty$. Since $v_n:=\ka u_n\in \Con_0$ for all $n\in \N$ with $v_n\searrow 0$ as $n\to \infty$, it follows that $\|u_n\|_\ka =\|v_n\|_\infty\to 0$ as $n\to \infty$ by Dini's lemma. In particular, the semigroup $\SS$ and in fact every continuous map $\UC_\ka\to \UC_\ka$ is continuous from above on $\UC_\ka$. 
  \item[c)] Assume that $\SS$ is continuous from above on $\Lipb$. the space $\Lipb$ is invariant under $\SS(t)$ for all $t\geq 0$. Note that $\SS(t)u\in \Lipb$ for all $u\in \Lipb$ and $t\geq 0$. Therefore, by \cite[Remark 5.4]{MR3824739}, $\SS(t)$ uniquely extends to an operator $\SS(t)\colon \Cb\to \Cb$, which is again continuous from above. Moreover, for every $n\in \N$, $v\in \Cb(M^{n+1})$ the mapping
  \[                                                                                                                                                                                                                              
   M^{n+1}\to \R,\quad (x_1,\ldots, x_n, x_{n+1})\mapsto \big(\SS(t)v(x_1,\ldots, x_n,\, \cdot \,)\big)(x_{n+1})
  \]
  is bounded and continuous.
 \end{enumerate}
\end{remark}

Continuity from above on $\Lipb$ will be crucial for the existence of a stochastic representation. In Remark \ref{rem.contabove} b), we have seen that, if $M$ satisfies the Heine-Borel property and $\ka\in \Con_0$, then $\SS$ is continuous from above on $\UC_\ka$. The following proposition, which is a generalisation of \cite[Proposition 2.8]{dkn}, gives a sufficient condition for the continuity from above on $\Lipb$ in the case that $\ka$ does not vanish at infinity and $M$ is (only) locally compact. Recall that $\Con_0$ is the closure of the space $\Lip_c$ of all Lipschitz continuous functions with compact support w.r.t.~the supremum norm $\|\cdot\|_\infty$, and that $\Con_0\subset \UC_{\rm b}\subset  \UC_\ka$.

\begin{proposition}\label{prop:A3}
 Suppose that for every $x\in M$ and every $\de>0$ there exists a function $\varphi_x\in \Con_0$ satisfying the following:
 \begin{enumerate}
  \item[(i)] $\varphi_x(x)=1$ and $0\leq \varphi_x\leq 1$,
  \item[(ii)] $\varphi_x\in \bigcap_{\la\in \La} D(A_\la)$ with $\sup_{\la\in \La}\|A_\la\varphi_x\|_\ka\leq \de$.
 \end{enumerate}
Then, $\SS$ is continuous from above on $\Lipb$.
\end{proposition}

\begin{proof}
 Fix $t> 0$, $x\in M$ and $\de>0$. Notice that $1\in D(A_\lambda)$ with $A_\lambda 1=0$ since $S_\lambda(t) 1 =1$ for all $\lambda\in \Lambda$. Therefore, $(1-\varphi_x)\in \bigcap_{\la\in \La}D(A_\la)$ with $A_\lambda(1-\varphi_x)=-A_\lambda \varphi_x$. Since $\varphi_x(x)=1$, it follows that
 \begin{align*}
  \ka(x)\big[\big(\SS(t)(1-\varphi_x)\big)(x)\big]&\leq \|\SS(t)(1-\varphi_x)-(1-\varphi_x)\|_\ka\leq te^{\al t}\sup_{\la\in \La}\|A_\la \varphi_x\|_\ka \\
  &\leq \de te^{\al t}.
 \end{align*}
 Let $(u_n)_{n\in \N}\subset \Lipb$ with $u_n\searrow 0$ as $n\to \infty$ and $\ep>0$. Then, there exists some $\varphi_x\in \Con_0$ satisfying (i) and (ii) with $\de=\frac{\ep \ka(x)}{2te^{\al t}c}$, where $c:=\max\big\{1,\|u_1\|_\infty\big\}$. Then,
 \[
  \|u_n\|_\infty\big(\SS(t)(1-\varphi_x)\big)(x)\leq \frac{\ep}{2}\quad \text{for all }n\in \N.
 \]
 Moreover, there exists some $n\in \N$ such that $\|u_n\varphi_x\|_\ka<\frac{\ep}{2}$ since $\varphi_x\in \Con_0$. Hence,
 \[
  \big(\SS(t)u_n\big)(x)\leq \|u_n\|_\infty\big(\SS(t)(1-\varphi_x)\big)(x)+\big(\SS(t)(u_n\varphi_x)\big)(x)<\ep.
 \]
 This shows that $\SS(t)u_n\searrow 0$ as $n\to \infty$. Now, let $(u_n)_{n\in \N}\subset \Lipb$ and $u\in \Lipb$ with $u_n\searrow u$ as $n\to \infty$. Then,
 \[
  |\SS(t)u_n-\SS(t)u|\leq \SS(t)(u_n-u)\searrow 0 \quad \text{as }n\to \infty.
 \]
\end{proof}

Note that, although not explicitly stated in Proposition \ref{prop:A3}, the existence of a function $\ph_x\in \Con_0$ with $\ph_x(x)\neq 0$ for all $x\in M$ implies that $M$ is locally compact. Thus, Proposition \ref{prop:A3} is thus only applicable for locally compact $M$. The following theorem is a direct consequence of \cite[Theorem 5.6]{MR3824739}.

\begin{theorem}\label{stochrep}
 Assume that $M$ is a Polish space and that $\SS$ is continuous from above on $\Lipb$. Then, there exists a quadruple $(\Om,\FF,(\EE^x)_{x\in M},(X_t)_{t\geq 0})$ such that
 \begin{enumerate}
  \item[(i)] $X_t\colon \Om\to M$ is $\FF$-$\BB$-measurable for all $t\geq 0$,
  \item[(ii)] $(\Om,\FF,\EE^x)$ is a sublinear expectation space with $\EE^x(u(X_0))=u(x)$ for all $x\in M$ and $u\in \Cb$,
  \item[(iii)] For all $0\leq s<t$, $n\in \N$, $0\leq t_1<\ldots <t_n\leq s$ and $v\in \Cb(M^{n+1})$,
  \begin{equation}\label{eq.stochrep1}
   \EE^x\big(v(X_{t_1},\ldots,X_{t_n},X_t)\big)=\EE^x\left(\big(\SS(t-s)v(X_{t_1},\ldots,X_{t_n},\, \cdot\,)\big)(X_s)\right).
  \end{equation}
 \end{enumerate}
 In particular,
 \begin{equation}\label{eq.stochrep2}
  \big(\SS(t)u\big)(x)=\EE^x(u(X_t)).
 \end{equation}
 for all $t\geq 0$, $x\in M$ and $u\in \Cb$.
\end{theorem}

\begin{remark}\
\begin{enumerate}
 \item[a)] The quadruple $(\Om,\FF,(\EE^x)_{x\in M},(X_t)_{t\geq 0})$ can be seen as a nonlinear version of a Markov process. As an illustration, we consider the case, where the semigroup $\SS$ and thus $\EE^x$ is linear for all $x\in M$, and choose $v=u(X_t)1_B(Y)$ with $u\in \UC_{\rm b}$ and $B\in \BB^n$, where $\BB^n$ denotes the product $\sigma$-algebra of the Borel $\sigma$-algebra $\BB$. Then, $\EE^x=\E_{\P^x}$ is the expectation w.r.t.\ a probability measure $\P^x$ on $(\Om,\FF)$ for all $x\in M$. Using the continuity from above and Dynkin's lemma, Equation \eqref{eq.stochrep1} reads as
\[
 \E_{\P^x}\big(u(X_t)1_B(X_{t_1},\ldots, X_{t_n})\big)=\E_{\P^x}\big[\big(\SS(t-s)u\big)(X_s)1_B(X_{t_1},\ldots, X_{t_n})\big],
\]
which is equivalent to the Markov property
\begin{equation}\label{eq.linmarkprop}
 \E_{\P^x}\big(u(X_t)|\FF_s\big)=\big(\SS(t-s)u\big)(X_s) \quad \P^x\text{-a.s.},
\end{equation}
where $\FF_s:=\sigma\big(\{X_u\, |\, 0\leq u\leq s\}\big)$. On the other hand, if $\EE^x=\E_{\P^x}$, the Markov property \eqref{eq.linmarkprop} implies Property (iii) from Theorem \ref{stochrep}. 
 \item[b)] A natural question, in particular in view of \eqref{eq.stochrep1} is, if the nonlinear expectation $\EE^x$ can be extended to unbounded functions satisfying a certain growth condition. We would like to point out that \cite[Theorem 5.6]{MR3824739} a priori only applies to bounded functions. Using the fact that $\EE^x$ admits a representation in terms of a nonempty set $\PP^x$ of probability measures on $(\Om,\FF)$, i.e.
 \[
  \EE^x(Y)=\sup_{\P\in \PP^x} \E_\P(Y)\quad \text{for all }Y\in \LL^\infty(\Om,\FF),
 \]
 allows to define
 \[
  \EE^x(Y):=\sup_{\P\in \PP^x} \E_\P(Y)\in \R
 \]
 for $\FF$-measurable functions $Y\colon \Om\to \R$ with $\sup_{\P\in \PP^x} \E_\P(|Y|)<\infty$. On the other hand, \eqref{eq.stochrep2} gives rise to a well-defined notion of $\EE^x$ for functions of the form $u(X_t)$ with $u\in \UC_\kappa$ and $t\geq 0$. Consider a weight function $w\in \UC_\kappa$ with $w(x)\geq 0$ for all $x\in M$ and a measurable function $u\colon M\to \R$ with $|u(x)|\leq w(x)$ for all $x\in M$. Then, \eqref{eq.stochrep2} implies that
 \[
 \EE^x\big(|u(X_t)|\big)\leq \EE^x\big(w(X_t)\big)=\big(\SS(t)w\big)(x)<\infty
 \]
 for all $t\geq 0$ and $x\in M$.
\end{enumerate}
\end{remark}


\section{Connection to control theory}\label{sec:control}

In this section, we discuss our results in light of the standard literature and standard examples in control theory. In particular, we discuss the relation between the semigroup envelope and the value function of Meyer-type control problems. We further go into more detail on our notion of a viscosity solution in view of the standard one and uniqueness results for the latter.

\subsection{The notion of viscosity solution and uniqueness}\label{sec:uniqueness}
A priori, our notion of a viscosity solution is somewhat different from the classical one related to (standard) parabolic HJB equations. The key difference between both notions is the class of test functions. While in a standard setting, the class of test functions typically consists of sufficiently smooth functions defined on the parabolic domain $[0,\infty)\times M$, in our notion, we formally separate the space and time variable and consider differentiable functions $\psi\colon [0,\infty)\to \UC_\kappa$ taking values in a function space. Here, time regularity is given in terms of differentiability in $t$ w.r.t.\ the norm $\|\cdot\|_\kappa$, and the convergence of the difference quotient to the derivative is thus up to the weight $\kappa$ uniform in the space variable. Space regularity is given in terms of the abstract condition $\psi(t)\in \DD$, where
\[
 \DD:=\bigg\{u\in \bigcap_{\la\in \La}D(A_\la)\, \bigg|\, \sup_{\la\in \La}\|A_\la u\|_\ka<\infty \text{ and } \lim_{h\searrow 0}\bigg\|\frac{\SS(h)u-u}{h}-\AA u\bigg\|_\ka= 0 \bigg\}.
\]
Let us consider as an illustrative example, the case where $M=\R$, $\kappa=1$, and $A_\lambda=\frac{\lambda^2}{2}\partial_{xx}$ for $\lambda\in \Lambda:=[\sigma_\ell,\sigma_h]$ with $0<\sigma_\ell\leq \sigma_h$. That is, our control parameter is the volatility of a Brownian Motion. In this case, $\DD=\UC_b^2$ is the space of all twice differentiable functions with bounded and uniformly continuous derivatives. We therefore see that, in the case of partial differential equations, the set $\DD$ typically encodes some sort of space regularity in terms of differentiability in the space variable. This will become also clear in the examples in Section \ref{sec:excontrol}.

As we point out in Remark \ref{rem.uniqueness}, it is, in general, unclear how rich the class of test functions for a viscosity solution from Definition \ref{def.viscosity} is. Therefore, uniqueness is not given a priori and has to be checked on a case by case basis. However, it is worth noting that, if $M$ is, for example, an open subset of $\R^d$ with $d\in \N$, the standard notion of a viscosity solution is very robust in view of the considered class of test functions, cf. Ishii \cite[Remark 1.5 and Example 1.2]{MR3135341}. Typically, one chooses functions that are twice differentiable on $M\times [0,\infty)$ with continuous derivatives up to order $2$ as test functions. However, the notion of a viscosity solution and, in particular, uniqueness is not affected by replacing $\Con^2(M\times [0,\infty))$, e.g., by $\Con^\infty_c([0,\infty)\times M)$, i.e. functions that are compactly supported and infinitely smooth functions. Roughly speaking this is due to the fact that the notion of a viscosity solution is a very local solution concept, and therefore only the local behaviour of test functions matters. We point out that under very mild conditions, e.g., for all $\de>0$ and $x\in M$, the existence of a cut-off function $\ph\in \DD$ with $0\leq \ph\leq 1$, $\ph(x)=0$, and $\ph(y)=1$ for $y\in M$ with $d(x,y)\geq \de$, our notion of a viscosity solution can also be formulated in terms of local extrema instead of global extrema; thus leading to a local solution concept as well.

We build on Remark \ref{rem.uniqueness} in the case that $M$ is an open subset of $\R^d$ with $d\in \N$. Assume that $\Con_c^\infty\subset \DD$, where $\Con^\infty_c$ denotes the space of all infinitely differentiable functions $M\to \R$ with compact support, and let $\psi\in \Con^\infty_c([0,\infty)\times M)$. Since $\psi$ has a compact support and $\kappa$ is continuous, it follows that
\[
 \sup_{x\in M}\kappa(x)\bigg|\frac{\psi(t+h,x)-\psi(t,x)}{h}-\partial_t\psi(t,x)\bigg| \leq Lh \to 0 \quad \text{as }h\searrow 0
\]
for all $t>0$. In particular, the function $$\psi\colon [0,\infty)\to \UC_\kappa, \; t\mapsto \psi(t):=\psi(t,\, \cdot\,)$$
is differentiable. Moreover $\psi(t)=\psi(t,\,\cdot\, )\in \Con_c^\infty\subset \DD$. Therefore, assuming that (at least) $\Con_c^\infty\subset \DD$, any $\psi\in \Con^\infty_c([0,\infty)\times M)$ is a test function in the sense of Definition \ref{def.viscosity}. Thus, the notion of a viscosity solution from Definition \ref{def.viscosity} coincides with the usual notion in most cases covered by the standard theory. As a consequence, uniqueness of viscosity solutions can be obtained from Ishii's lemma.

\subsection{Semigroup envelopes as value functions to optimal control problems}

In this section, we identify the semigroup envelope as the value function of a space-time discrete Meyer-type optimal control problem under the additional assumption that each semigroup $S_\lambda$ is a family of transition kernels of a stochastic process. In the following, we describe the broad idea behind the approach using semigroup envelopes. Assume that, $S_\lambda$ is a semigroup of transition kernels of a controlled stochastic process $(X_t^{x, \lambda})_{t\geq}$ (for the sake of a simplified notation defined on the same probability space) with control set $\Lambda$ and control parameter $\lambda\in \Lambda$, i.e.
\[
 \big(S_\lambda(t)u\big)(x)=\E\big[u\big(X_t^{\la,x}\big)\big]
\]
for $x\in M$, $\lambda\in \Lambda$, $t\geq 0$, and $u\in \UC_\kappa$. Then, for a fixed time-horizon $t\geq 0$, one typically considers a (suitably defined) set of admissible controls $\Lambda_{\rm ad}^t$ and the value function 
\begin{equation}\label{eq.optcont}
 V(u,t,x):=\sup_{\lambda\in \Lambda_{\rm ad}^t}\E\big[u\big(X_t^{\la,x}\big)\big]
\end{equation}
of the related Meyer-type optimal control problem. Note that this is usually only possible if the controlled dynamics satisfy a certain structure. The idea behind the semigroup envelope is to transform the dynamic optimization problem given in terms of the value function \eqref{eq.optcont} into a series of static optimization problems with value functions of the form
\begin{equation}\label{eq.static}
  \sup_{\lambda\in \Lambda}\E\big[u\big(X_t^{\la,x}\big)\big]=\sup_{\lambda\in \Lambda}\big(S_\la(t)u\big)(x)=:\big(\EE_t u\big)(x),
\end{equation}
Now, one considers a partition $\pi=\{t_0,\ldots, t_m\}\in P_t$ with $0=t_0<\ldots <t_m=t$ of the time-interval $[0,t]$, and one optimizes after each time-step, leading to the expression
\begin{equation}\label{eq.static1}
  \big(\EE_\pi u\big)(x)\big(\EE_{t_1-t_0}\cdots \EE_{t_m-t_{m-1}} u\big)(x).
\end{equation}
 Letting the mesh size $|\pi|$ of the partition $\pi$ tend to zero or taking the supremum over all partitions $\pi\in P_t$ leads to a formal approximation of the dynamic optimization problem \eqref{eq.optcont} in terms of a series of static control problems on a grid that becomes finer and finer as the mesh size tends to zero.
 
 In the sequel, we will make this approximation rigorous by choosing the set of admissible controls as space-time discrete controls. To that end, we consider static controls of the form
 \[
  \Lambda_M:=\bigg\{(\lambda_i, B_i)_{i\in \N}\in (\Lambda\times \BB)^\N\, \bigg|\, B_i\cap B_j=\es\text{ for }i\neq j\text{ and }\bigcup_{i\in \N}B_i=M\bigg\}
 \]
One can think of $\lambda=(\lambda_i, B_i)_{i\in \N}\in (\Lambda\times \BB)^\N\in \Lambda_M$ as a function taking the value $\lambda_i$ on $B_i$ for each $i\in \N$. For $\lambda=(\lambda_i,B_i)_{i\in \N}\in (\Lambda\times \BB)^\N\in \Lambda_M$, we define
\begin{equation}\label{sqi}
\big(S_\la(t) u\big)(x):=\sum_{i\in \N}1_{B_i}(x)\big(S_{\la^i}(t)u\big)(x)
\end{equation}
for all $x\in M$ and $u\in \UC_\kappa$. We now add a dynamic component, and define
\begin{equation}\label{eq.admiss}
 \Lambda_{\rm ad}^t:=\bigg\{ (\la_k,h_k)_{k=1,\ldots, m}\in \big(\La_M\times [0,t]\big)^m\, \bigg|\,  m\in \N,\; \sum_{k=1}^m h_k=t\bigg\}.
\end{equation}
Roughly speaking, the set $\Lambda_{\rm ad}^t$ corresponds to the set of all space-time discrete admissible controls for the control set $\La$. For $\lambda=(\la_k,h_k)_{k=1,\ldots, m}\in \Lambda_{\rm ad}^t$ with $m\in \N$ and $u\in \UC_\kappa$, we define
\[
 J_\lambda u:=S_{\la_1}(h_1)\cdots S_{\la_m}(h_m)u,
\]
 where $S_{\la_k}(h_k)$ is defined as in \eqref{sqi} for $k=1,\ldots, m$. Then, for all $t\geq 0$, $u\in \UC_\kappa$, and $x\in M$, 
\begin{equation}\label{primaldual}
 \big(\SS(t)u\big)(x)=\sup_{\lambda\in \Lambda_{\rm ad}^t}\big(J_\lambda u\big)(x).
\end{equation}
That is, the semigroup envelope is the value function of an abstract analogue of the optimal control problem \eqref{eq.optcont} with $\Lambda_{\rm ad}^t$ given as in \eqref{eq.admiss}. In fact, by definition of $\Lambda_{\rm ad}^t$, it follows that $\sup_{\lambda\in \Lambda_{\rm ad}^t}J_\lambda u\leq \SS(t)u$ for all $t\geq 0$ and $u\in \UC_\ka$. On the other hand, let $\ep>0$ and $\pi=\{t_0,\ldots, t_m\}\in P_t$ with $0=t_0<\ldots< t_m=t$, and define $h_k:=t_k-t_{k-1}$ for $k =1,\ldots, m$. By a backward recursion, we may choose an $\frac{\ep}{2m}$-optimizer of $\EE_{h_k}\cdots \EE_{h_m} u$ for each $x\in M$ and $k=m,\ldots, 1$. Since $M$ is separable, there exist $\la_1,\ldots, \lambda_m\in \La_M$ such that $$\EE_\pi u=\EE_{h_1}\cdots \EE_{h_m} u\leq S_{\la_1}(h_1)\cdots S_{\la_m}(h_m)u+\ep= J_\lambda u+\ep$$ where $\lambda:=(\la_k,t_k-t_{k-1})_{k=1,\ldots, m}\in \Lambda_{\rm ad}^t$. Letting $\ep\to 0$ and taking the supremum over all $\pi\in P_t$ and $\lambda\in \Lambda_{\rm ad}^t$, yields $\SS(t)u\leq \sup_{\lambda\in \Lambda_{\rm ad}^t}J_\lambda u$.

Considering standard cases in optimal control, the connection between semigroup envelopes and the value function of a Meyer-type optimal control problem can also be established a posteriori, since both lead to a viscosity solution to the same HJB-equation. In these cases, one thus sees that the optimizing over space-time discrete admissible controls, which we have discussed in this section, is equivalent to optimizing over usual admissible controls, which typically possess a nondiscrete structure.

\subsection{Some illustrative examples from control theory}\label{sec:excontrol}

In this section, we discuss two examples in the context of control theory. For $k\in \N_0$, let $\Lipb^k$ denote the space of all $k$-times differentiable functions with bounded and Lipschitz continuous (Fr\'echet) derivatives up to order $k$.

\begin{example}[Geometric Brownian Motion]
 Let $M=\R$ and $\La$ be a nonempty set of tuples $(\mu,\si)\subset \R\times [0,\infty)$ with
 \[
  \beta:=\sup_{(\mu,\si)\in \La}|\mu|+\frac{\si^2}{2}<\infty
 \]
 Let $\la=(\mu,\si)\in \La$, $p\geq 1$, and $W$ be a Brownian Motion on a probability space $(\Om,\FF,\P)$. Define
 \[
  X_t^\la:=\exp\bigg(t\big(\mu -\tfrac{\si^2}{2}\big) +\si W_t\bigg)
 \]
 for $t\geq 0$ and $x\in \R$. Then,
 \[
  \E(|X_t^\la|^p)^{\tfrac{1}{p}}= e^{\left(\mu+\tfrac{(p-1)\si^2}{2}\right) t}\leq e^{p \be t}.
 \]
  Moreover,
\begin{align*}
  \E\big(|X_t^\la-1|^2\big)&=1-2\E\big(X_t^\la\big)+\E\big(|X_t^\la|^2\big)=1-2e^{\mu t}+e^{(2\mu +\sigma^2)t}\\
  & \leq 1-2e^{-\be t}+e^{2\be t}.
 \end{align*}
 Let $\ka(x):=(1+|x|)^{-p}$ for $x\in \R$ and $S_\la$ be given by
 \[
  \big(S_\la(t)u\big)(x):=\E\big(u(x X_t^\la)\big)
 \]
 for $u\in \UC_\ka$, $t\geq 0$, and $x\in \R$. Then, it follows that $\|S_\lambda(t)u\|_\ka\leq e^{p \be t} \|u\|_\ka$ for $t\geq 0$ and $u\in \UC_\ka$. Moreover, for $u\in \Lipb$, $\|u\|_{\Lip}\leq e^{\be t}\|u\|_{\Lip}$ and
 \begin{equation}\label{eq:GBM2}
  \|S_\lambda(t)u-u\|_\ka \leq \|u\|_{\Lip}\E \big(|X_t^1-1|\big)\leq \sqrt{1-2e^{-\be t}+e^{2\be t}}\to 0\quad \text{as }t\to 0.
 \end{equation}
 Therefore, by Theorem \ref{main1} and Proposition \ref{scontin}, the semigroup envelope $\SS$ for the family $(S_\la)_{\la\in \La}$ exists and is a strongly continuous Feller semigroup. Let $u\in \Lipb^2$ with compact support $\supp(u)$ and 
 $A_\la u\in \Lipb$ be given by
 \[
  \big(A_\la u\big)(x):= \mu xu'(x) +\frac{\si^2x^2}{2}u''(x)\quad \text{for }x\in \R.
 \]
 Since $\supp(u)$ is compact,
 \[
  \sup_{\la\in \La}\|A_\la u\|_\infty<\infty\quad \text{and}\quad  C_u:=\sup_{\la\in \La}\|A_\la u\|_\Lip<\infty.
 \]
 By Ito's formula, it follows that
 \[
  \frac{\big(S_\la(h)u\big)(x)-u(x)}{h}=\frac{1}{h}\int_0^h \big(S_\la(s)A_\la u\big)(x)\, {\rm d}s
 \]
 for all $h>0$ and $x\in \R$, which, together with \eqref{eq:GBM2}, implies that
 \[
  \bigg\|\frac{S_\la(h)u-u}{h}-A_\la u\bigg\|_\ka \leq C_u\sqrt{1-2e^{-\be h}+e^{2\be h}}\to 0\quad \text{as }h\searrow 0.
 \]
 It follows that the set of all $u\in \Lipb^2$ with compact support $\supp(u)$ is contained in $\DD$. By Theorem \ref{viscosity1}, we thus obtain that $u(t):=\SS(t)u_0$, for $t\geq 0$, defines a viscosity solution to the fully nonlinear Cauchy problem
 \begin{eqnarray*}
  \partial_t u(t,x)&=&\sup_{(\mu,\si)\in \La} \left(\mu x\partial_xu(t,x)+\frac{\si^2x^2}{2}\partial_{xx}u(t,x)\right), \quad (t,x)\in (0,\infty)\times \R,\\
  u(0,x)&=&u_0(x),\quad x\in \R.
 \end{eqnarray*}
 Under the nondegeneracy condition $\inf_{(\mu,\sigma)\in \Lambda} |\sigma|>0$, the above HJB equation has a unique viscosity solution. By Remark \ref{rem.contabove} b), the semigroup $\SS$ is continuous from above. The nonlinear Markov process related to $\SS$ can be seen as a geometric $G$-Brownian Motion (cf. Theorem \ref{stochrep}).
\end{example}

\begin{example}[Ornstein-Uhlenbeck processes on separable Hilbert spaces]\label{ex:OU}
 We consider the case where $M=H$ is a real separable Hilbert space. Let $\La$ be a set of triplets $(B,m,C)$, where $m\in H$, $B\in L(H)$, and $C\in L(H)$ is a self-adjoint positive semidefinite trace class operator, with
 \[
  \be:=\sup_{(B,m,C)\in \La} \big(\|B\| + \|m\|+\| C\|_{\tr}\big)<\infty.
 \]
 Let $\lambda=(B,m,C)\in \Lambda$, $T_B(t):=e^{tB}\in L(H)$, for $t\geq 0$, and $W^C$ be an $H$-valued Brownian Motion with covariance operator $C$ on a probability space $(\Om,\FF,\P)$. For $t\geq 0$, we define
 \[
  X_t^\lambda:=\int_0^t T_B(s)m\, {\rm d}s+\int_0^t T_B(t-s)\, {\rm d}W^C_s
 \]
  and $S_\la$ by
 \[
  \big(S_\la (t)u\big)(x):=\E\big(u(T_B(t)x+X_t^\la)\big)
 \]
 for $x\in H$, $t\geq 0$, and $u\in \UC_\ka$. Moreover, let $\ka:=(1+\|x\|^2)^{-1}$ for $x\in H$. Using basic facts from (infinite-dimensional) stochastic calculus and \eqref{ex:koop1}, below, for $F(x)=B(x)+m$,
 \begin{align*}
   1+\E(\|T_B(t)x+X_t^\lambda\|^2)&\leq 1+\bigg\|T_B(t)x+\int_0^t T_B(s)m\, {\rm d}s\bigg\|^2+\int_0^t e^{2\|B\|s}\|C\|_{\tr}\, {\rm d}s\\
   &\leq (1+\|x\|^2)e^{2\big(\|B\|+\|m\|\big)t}+e^{2\|B\|t}\|C\|_{\tr}t\\
  &\leq (1+\|x\|^2)e^{2\be t}
 \end{align*}
 for all $t\geq 0$ and $x\in H$, which implies that $\|S_\la(t) u\|_{\ka}\leq e^{2\be t }\|u\|_\ka$ for all $t\geq 0$ and $u\in \Lipb$. By \eqref{ex:koop2}, below, $\|S_\la(t) u\|_\Lip\leq e^{\be t}\|u\|_\Lip$ for all $t\geq 0$ and $u\in \Lipb$. For $u\in \Lipb^2$, let
 \[
  C_u:=\max\{\|D_x u\|,\|D_x^2 u\|_\infty,\|D_x^2 u\|_{\Lip}\},
 \]
 where $D_x$ and $D_x^2$ denote the first and second Fr\'echet derivative in the space-variable, and $A_\la u\in C_\ka$ be given by
 \[
  \big(A_\la u\big)(x) =D_x u(x)(Bx+m)+\frac{1}{2}\tr\big(CD_x^2u(x)\big)
 \]
 for $x\in H$. Then, for all $h\geq 0$ and $x\in H$,
 \begin{align*}
    \big|\big(S_\la (h)A_\la  u\big)(x)-\big(A_\la u\big)(x)\big|\leq  C_u \be (1+\|x\|) \E\big(\|T_B(t)x+X_t^\la-x\|\big).
 \end{align*}
We estimate the last term using \eqref{ex:koop1}, below, and obtain that
 \begin{align*}
  \E\big(\|T_B(t)x+X_t^\la-x\|\big)&\leq \big(e^{(\|B\|+\|m\|)t}-1\big)\big(1+\|x\|\big)+\sqrt{\|C\|_{\tr}t}\\
  &\leq (1+\|x\|)\big(e^{\be t}-1+\sqrt{\be t}\big).
 \end{align*}
 Therefore, 
 \[
  \|S_\la (h)A_\la  u-A_\la u\|_\ka \leq C_u \sqrt{2}\be \big(e^{\be h}-1+\sqrt{\be h}\big).
 \]
 By Ito's formula, it follows that
 \[
  \frac{\big(S_\la (h)u\big)(x)-u(x)}{h}=\frac{1}{h}\int_0^h \big(S_\la (s)Au\big)(x)\, {\rm d}s
 \]
 for all $h>0$ and $x\in H$, which implies that
 \[
  \bigg\|\frac{S_\la (h)u-u}{h}-A_\la u\bigg\|_\ka \leq C_u \sqrt{2}\be \big(e^{\be h}-1+\sqrt{\be h}\big)\to 0 \quad \text{as }h\searrow 0.
 \]
 In order to show that $\Lipb^2\subset \DD$, it remains to show that $\SS$ is strongly continuous. For this we invoke Proposition \ref{critstrongcont2}. Note that $\Lipb^2$ is not dense in $\Lipb$ if $H$ is infinite-dimensional. Let $\de\in (0,1]$ and $\varphi\colon [0,\infty)\to [0,1]$ infinitely smooth with $\varphi(s)=1$ for $x\in \big[0,\tfrac{\de}{2}\big]$ and $\varphi(s)=0$ for $s\in [\de,\infty)$. For $x,y\in H$, let $\varphi_x(y):=\varphi(\|y-x\|)$. Then, $\varphi_x\in \Lipb^2$ with
 \[
 \|D_x\varphi_x \|_\infty\leq \|\varphi'\|_\infty\quad \text{and}\quad \|D_x^2\varphi_x\|_\infty\leq \frac{3}{\de}\|\varphi'\|_\infty+\|\varphi''\|_\infty\quad  \text{for all }x\in M.
 \]
 Hence,
 \[
  \|A_\la \varphi_x\|_\ka\leq \frac{5\be }{2\de}\max\big(\|\varphi'\|_\infty+\|\varphi''\|_\infty\big)=:L
 \]
 for all $x\in M$. Therefore, by Proposition \ref{critstrongcont2}, the semigroup $\SS$ is strongly continuous. Altogether, we have shown that the assumptions (A1) and (A2) are satisfied, the semigroup envelope $\SS$ is strongly continuous and $\Lipb^2\subset \DD$. By Theorem \ref{viscosity1}, we thus obtain that $u(t):=\SS(t)u_0$, for $t\geq 0$, defines a viscosity solution to the fully nonlinear PDE
 \begin{eqnarray*}
  \partial_t u(t,x)&=&\sup_{(B,m,C)\in \La}\left( D_x u(t,x)(Bx+m)+\frac{1}{2}\tr\big(CD_x^2u(t,x)\big)\right),\\
  &&\qquad \qquad\qquad \qquad\qquad \qquad\qquad \qquad \qquad\quad (t,x)\in (0,\infty)\times H,\\
  u(0,x)&=&u_0(x),\quad x\in H.
 \end{eqnarray*}
 We point out that, by Remark \ref{rem.uniqueness}, the class of test functions in the definition of a viscosity solution contains the set $\Lipb^{1,2}[0,\infty)\times H)$. 
 If $H=\R^d$, the semigroup $\SS$ is continuous from above by Remark \ref{rem.contabove} b), which implies the existence of an O-U-process under a nonlinear expectation which represents $\SS$ (cf. Theorem \ref{stochrep}).
\end{example}
 
\section{Further examples}\label{sec:ex}
For $k\in \N_0$, let $\Lipb^k$ denote the space of all $k$-times differentiable functions with bounded and Lipschitz continuous derivatives up to order $k$.

\begin{example}[Koopman semigroups on real separable Banach spaces]
 We consider the case, where the state space $M=X$ is a real separable Banach space. We denote topological dual space of $X$ by $X'$ and the operator norm on $X'$ by $\|\cdot\|_{X'}$. We consider a nonempty set $\La$ of Lipschitz continuous functions $F\colon X\to X$ with
 \[
  \be:=\sup_{F\in \La}\bigg(\sup_{\substack{x,y\in M\\ x\neq y}}\frac{\|F(x)-F(y)\|}{\|x-y\|}\bigg)<\infty\quad \text{and}\quad \al:=\be+\sup_{F\in \La}\|F(0)\|<\infty.
 \]
 Let $F\in \La$, and denote by $\Phi_F\colon [0,\infty)\times X\to X$ the \textit{continuous semiflow} related to the ODE $x'=F(x)$, i.e., for $x\in X$, $\Phi(\,\cdot\,,x)$ is the unique solution to the initial value problem
 \begin{align}
 \partial_t\Phi_F(t,x)&=F\big(\Phi_F(t,x)\big),\quad \text{for }t\geq 0,\\
 \Phi_F(0,x)&=x.
 \end{align}
 Then, by Gronwall's lemma,
 \begin{equation}\label{ex:koop2}
  \|\Phi_F(t,x)-\Phi_F(t,y)\|\leq e^{\be t}\|x-y\|
 \end{equation}
 for all $t\geq 0$ and $x,y\in X$. Moreover,
 \[
  1+\|x\|+\|\Phi_F(t,x)-x\|\leq 1+\|x\|+\al \int_0^t 1+\|x\|+\|\Phi_F(s,x)-x\|\, {\rm d}s
 \]
 for all $t\geq 0$ and $x\in X$. Again, by Gronwall's lemma, it follows that
 \begin{equation}\label{ex:koop1}
  1+\|\Phi_F(t,x)\|\leq 1+\|x\|+\|\Phi_F(t,x)-x\|\leq \big(1+\|x\|\big)e^{\al t}
 \end{equation}
 for all $t\geq 0$ and $x\in X$. Let $p\in (0,\infty)$ and $\ka(x):=\big(1+\|x\|\big)^{-p}$ for all $x\in X$. For $u\in \UC_\ka$, $t\geq 0$, and $x\in X$, we define
 \[
  \big(S_F(t)u\big)(x):=u\big(\Phi_F(t,x)\big).
 \]
 Then, by \eqref{ex:koop1}, for $u\in \UC_\ka$, $t\geq 0$, and $x\in X$,
 \[
  \big|\big(S_F(t)u\big)(x)\big|\leq \|u\|_\ka\big(1+\|\Phi_F(t,x)\|\big)^p\leq \|u\|_\ka\big(1+\|x\|\big)^pe^{\al p t},
 \]
 which implies that $\|S_F(t)u\|_\ka\leq e^{\al pt}\|u\|_\ka$. Moreover, Equation \eqref{ex:koop2} yields that $\|S_F(t)u\|_{\Lip}\leq e^{\be t}\|u\|_{\Lip}$ for all $u\in \Lipb$. We have therefore shown that the family of semigroups $(S_F)_{F\in \La}$ satisfies the assumptions (A1) and (A2), so that the semigroup envelope of the family $(S_F)_{F\in \La}$ exists. We continue by showing that the semigroup envelope $\SS$ is strongly continuous. Let $\theta:=\min\{1,p\}$ and $u\in \UC_b$ with
 \[
  C_{u,\theta}:=\sup_{x,y\in M}\frac{|u(x)-u(y)|}{|x-y|^\theta}<\infty.
 \]
 Again, by \eqref{ex:koop1},
 \[
  \big|\big(S_F(t)u\big)(x)-u(x)\big|\leq C_{u,\theta} \|\Phi_F(t,x)-x\|^\theta\leq C_{u,\theta}\big(1+C+\|x\|\big)^\theta\big(e^{\al t}-1\big)^\theta.
 \]
 Therefore,
 \begin{equation}\label{ex:koop3}
   \|S_F(t)u-u\|_\ka\leq C_{u,\theta} \big(e^{\al t}-1\big)^\theta.
 \end{equation}
 Since the set of all H\"older continuous functions of degree $\theta$ is dense in $\UC_b$ w.r.t. $\|\cdot\|_\infty$ (and consequently w.r.t. $\|\cdot\|_\kappa$), and $\UC_b$ is dense in $\UC_\kappa$ w.r.t. $\|\cdot \|_\kappa$, Proposition \ref{scontin} implies that the semigroup envelope $\SS$ is strongly continuous.
 Let $u\in \Lipb^1$ with bounded support $\supp( u):=\overline{\{x\in X\, |\, u(x)\neq 0\}}\subset X$, i.e. $\supp( u)\subset B(0,R)$ for some $R>0$, and let $A_Fu\in \Lipb$ be given by
 \[
  \big(A_Fu\big)(x):=u'(x)F(x)\quad \text{for }x\in X,
 \]
 where $u'\colon X\to X'$ denotes the (first) Fr\'echet derivative of $u$. Since $\supp( u)$ is bounded,
 \[
  \sup_{F\in \La}\|A_Fu\|_\infty<\infty \quad \text{and}\quad C_u:=\sup_{F\in \La} \|A_Fu\|_\Lip <\infty.
 \]
 By the chain rule and the fundamental theorem of infinitesimal calculus, it follows that
 \[
  \frac{\big(S_F(h)u\big)(x)-u(x)}{h}=\frac{1}{h}\int_0^h \big(S_F(s)A_Fu\big)(x)\, {\rm d}s
 \]
 for all $h>0$ and $x\in X$, which, together with \eqref{ex:koop3}, implies that
 \[
  \bigg\|\frac{S_F(h)u-u}{h}-A_Fu\bigg\|_\ka\leq C_u\big(e^{\al h}-1\big)\to 0\quad \text{as }h\searrow 0.
 \]
 Hence, $\DD$ contains the set of all $u\in \Lipb^1$ with bounded support $\supp (u)$. By Theorem \ref{viscosity1}, we thus obtain that $u(t):=\SS(t)u_0$, for $t\geq 0$, defines a viscosity solution to the fully nonlinear PDE
 \begin{eqnarray*}
  \partial_t u(t,x)&=&\sup_{F\in \La} D_x u(t,x)F(x), \quad (t,x)\in (0,\infty)\times X,\\
  u(0,x)&=&u_0(x),\quad x\in X,
 \end{eqnarray*}
 where $D_x$ denotes the (first) Fr\'echet derivative in the space-variable. If $X=\R^d$, the semigroup envelope $\SS$ is continuous from above by Remark \ref{rem.contabove} b). In this case, Theorem \ref{stochrep} implies the existence of a Markov process under a nonlinear expectation related to $\SS$. This Markov process can be viewed as a nonlinear drift process.
\end{example}

\begin{example}[L\'evy Processes on abelian groups]
 Let $M=G$ be an abelian group with a translation invariant metric $d$ and $\ka(x):=1$ for all $x\in M$. Let $(S(t))_{t\geq 0}$ be a Markovian convolution semigroup, i.e. a semigroup arising from a L\'evy process. Then, $(S(t))_{t\geq 0}$ is a strongly continuous Feller semigroup of linear contractions (cf. \cite{dkn}). Moreover, due to the translation invariance, $\|S(t)u\|_{\Lip}\leq \|u\|_{\Lip}$ for all $t\geq 0$ and $u\in \Lipb$. Now, let $(S_\la)_{\la\in \La}$ be a family of Markovian convolution semigroups with generators $(A_\la)_{\la\in \La}$. Then, the assumptions (A1) - (A2) are satisfied. We refer to \cite{dkn} for examples, where the semigroup envelope is strongly continuous. In particular, all examples from \cite{dkn} fall into our theory. In the case, where $G=H$ is a real separable Hilbert space, we can improve the result obtained in \cite[Example 3.3]{dkn}. In this case, by the L\'evy-Khintchine formula (see e.g. \cite[Theorem 5.7.3]{MR874529}), every generator $A$ of a Markovian convolution semigroup is characterized by a L\'evy triplet $(b,\Si,\mu)$, where $b\in H$, $\Si\in L(H)$ is a self-adjoint positive semidefinite trace-class operator and $\mu$ is a L\'evy measure on $H$. For $u \in \Lipb^2(H)$ and a L\'evy triplet $(b,\Si,\mu)$, the generator $A_{b,\Si,\mu}$ is given by
 \begin{align*}
  \big(A_{b,\Si,\mu}u\big)(x)= \langle &b, D_x u(x)\rangle + \frac{1}{2}{\rm tr} \big(\Si D_x^2u(x)\big)\\
  &+\int_{H} u(x+y)-u(x)-\langle D_x u(x), h(y)\rangle\, {\rm d}\mu(y)
 \end{align*}
 for $x\in H$. Here, $D_x$ and $D_x^2$ denote the first and second Fr\'echet derivative in the space-variable, respectively, and the function $h\colon H\to H$ is defined by $h(y)=y$ for $\|y\|\leq 1$ and $h(y)=0$ whenever $\|y\| > 1$. Let $\La$ be a nonempty set of L\'evy triplets. We assume that
 \begin{equation}\label{lev1}
  C:=\sup_{(b,\Si,\mu)\in \La} \bigg(\|b\|+\|\Si\|_{\tr }+\int_{H} 1\wedge \|y\|^2\, {\rm d}\mu(y)\bigg)<\infty.
 \end{equation}
 Note that \eqref{lev1} does not exclude any L\'evy triplet a priori. Under \eqref{lev1}, the semigroup envelope $\SS$ is strongly continuous on $\Lipb^2$. In order to show that $\Lipb^2\subset \DD$, by the computations in \cite{dkn}, it suffices to show that $\SS$ is strongly continuous. For this we invoke Proposition \ref{critstrongcont1}. For $\de>0$, we choose the family $\big(\varphi_x)_{x\in H}$ as in the previous example. Since $\big(\SS(t)v\big)(x)=\big(\SS(t)v(x+\cdot )\big)(0)$ for all $v\in \UC_\ka$, $x\in H$ and $t\geq 0$, it follows that
 \[
  \big(\SS(t)(1-\varphi_x)\big)(x)=\big(\SS(t)(1-\varphi_0) \big)(0)
 \]
  for all $x\in H$ and $t\geq 0$. Defining $f(t):=\big(\SS(t)(1-\varphi_0) \big)(0)$ for $t\geq 0$, it follows that $f$ is continuous with $f(0)=0$. Therefore, by Proposition \ref{critstrongcont1}, the semigroup $\SS$ is strongly continuous. Altogether, we have shown that under the condition \eqref{lev1}, the assumptions (A1) and (A2) are satisfied, the semigroup envelope $\SS$ is strongly continuous and $\Lipb^2\subset \DD$.
  By Theorem \ref{viscosity1}, we thus obtain that $u(t):=\SS(t)u_0$, for $t\geq 0$, defines a viscosity solution to the fully nonlinear Cauchy problem
  \begin{eqnarray*}
  u_t(t,x)&=&\sup_{(b,\Si,\mu)\in \La} \big(A_{b,\Si,\mu} u(t)\big)(x), \quad (t,x)\in (0,\infty)\times H,\\
  u(0,x)&=&u_0(x),\quad x\in H.
 \end{eqnarray*}
  If $H=\R^d$ and the set of L\'{e}vy measures within the set of L\'evy triplets $\La$ is tight, Proposition \ref{prop:A3} implies that the semigroup envelope $\SS$ is continuous from above, leading to the existence of a nonlinear L\'{e}vy process related to $\SS$. However, due to the translation invariance of the semigroups, the continuity from above is actually not necessary in order to obtain the existence of a L\'{e}vy process under a nonlinear expectation. The nonlinear L\'{e}vy process can be explicitly constructed via space-time discrete stochastic integrals w.r.t. L\'{e}vy processes with L\'evy triplet contained in $\La$. We refer to \cite[Proposition 5.12]{dkn} for the details of the construction.
\end{example}

\begin{example}[$\al$-stable L\'evy processes]
Consider the setup of the previous example, with $G=\R^d$ for some $d\in \N$ and let $A_\al:=-(-\De)^\al$ be fractional Laplacian for $0<\al<1$. Then, for any compact subset $\La\subset (0,1)$, condition \eqref{lev1} is satisfied. Hence, the assumptions (A1) and (A2) are satisfied and the semigroup envelope $\SS$ is strongly continuous with $\Lipb^2\subset \DD$. By Theorem \ref{viscosity1}, we thus obtain that $u(t):=\SS(t)u_0$, for $t\geq 0$, defines a viscosity solution to the nonlinear Cauchy problem
  \begin{eqnarray*}
  u_t(t,x)&=&\sup_{\al\in \La} -(-\De)^\al u(t,x), \quad (t,x)\in (0,\infty)\times \R^d,\\
  u(0,x)&=&u_0(x),\quad x\in \R^d.
 \end{eqnarray*}
The related nonlinear L\'evy process can be interpreted as a $\La$-stable L\'evy process.
\end{example}

\begin{example}[Mehler semigroups]
 Consider the case, where the state space $M=H$ is a real separable Hilbert space and $\ka=1$. Let $(T,\mu)$ be a tuple consisting of a $C_0$-semigroup $T=(T(t))_{t\geq 0}$ of linear operators on $H$ with $\|T(t)\|\leq e^{\al t}$ for all $t\geq 0$ and some $\al\in \R$ and a family $\mu=(\mu_t)_{t\geq 0}$ of probability measures on $H$ such that
 \[
  \mu_0=\de_0\quad \text{and}\quad \mu_{t+s}=\mu_s \ast\mu_t\circ T(s)^{-1}\quad \text{for all }s,t\geq 0.
 \]
 We then define the generalized Mehler semigroup $S=S_{(T,\mu)}$ by 
 \[
  \big(S(t)u\big)(x):=\int_H u(T(t)x+y)\, {\rm d}\mu_t(y)
 \]
 for $u\in \UC_{\rm b}$, $t\geq 0$ and $x\in H$, see e.g.~\cite{MR1392452},\cite{MR1745332}. Then, $\|S(t)u\|_\infty\leq \|u\|_\infty$ for all $u\in \Cb$ and $\|S(t)u\|_{\Lip}\leq e^{\al t}\|u\|_{\Lip}$ for $u\in \Lipb$.
 Hence, for any nonempty family $\La$ of tuples $(T,\mu)$ with $\|T(t)\|\leq e^{\al t}$ for all $t\geq 0$ the assumptions (A1) and (A2) are satisfied.
\end{example}

\begin{example}[Bounded generators on $\ell^\infty$]
 Let $M=\N$ and $\ka(i)=1$ for all $i\in \N$. Let $(A_\la)_{\la\in \La}\subset L(\ell^\infty)$ be a family of operators satisfying the positive maximum principle and
 $$\sup_{\la\in \La}\|A_\la\|_{L(\ell^\infty)}<\infty.$$
 Here, we say that an operator $A\in L(\ell^\infty)$ satisfies the positive maximum principle if $A_{ii}<0$ for all $i\in \N$ and $A_{ij}\geq 0$ for all $i,j\in \N$ with $i\neq j$. Then, the family $(A_\la)_{\la\in \La}$ satisfies the assumptions (A1) and (A2) with $\DD=\ell^\infty$. In particular, the semigroup envelope is strongly continuous. If $A_\la 1=0$ for all $\la\in \La$, then the semigroup envelope admits a stochastic representation. This representation can be seen as a nonlinear Markov chain with state space $\N$.
\end{example}

\begin{example}[Multiples of generators of Feller semigroups]
 Let $A$ be the generator of a strongly continuous Feller semigroup $(S(t))_{t\geq 0}$ of linear operators. Assume that there exist constants $\al,\be\in \R$ such that
  \[
  \|S(t)u\|_\ka\leq e^ {\al t}\|u\|_\ka\quad \text{and}\quad \|S(t)u\|_{\Lip}\leq e^{\be t}\|u\|_{\Lip}
  \]
 for all $u\in \Lipb$ and $t\geq 0$. For $\la\geq 0$ let $A_\la:=\la A$ for all $\la$. Then, $A_\la$ generates the semigroup $S_\la$ given by $S_\la(t):=S(\la t)$ for all $t\geq 0$ and $\la\geq 0$. Then, for any compact set $\La\subset [0,\infty)$ the family $(S_\la)_{\la\in \La}$ satisfies the assumptions (A1) and (A2) with $D(A)\subset \DD$ and the semigroup envelope is strongly continuous. Hence, by Theorem \ref{viscosity1}, we obtain that $u(t):=\SS(t)u_0$, for $t\geq 0$, defines a viscosity solution to the abstract Cauchy problem
 \begin{eqnarray*}
   u'(t)&=&\sup_{\la\in \La}\la A u(t), \quad \text{for }t> 0,\\
  u(0)&=&u_0.
 \end{eqnarray*}
\end{example}

\end{document}